\documentclass[12pt]{amsart}
\usepackage[top=30truemm, bottom=30truemm, left=25truemm, right=25truemm]{geometry}
\allowdisplaybreaks[4]
\usepackage{tikz}
\usepackage{url}
\usepackage{amssymb}
\usepackage{enumerate}
\usepackage{braket}
\usepackage{amsmath, systeme}
\usepackage{bm}

\theoremstyle{definition}
\newtheorem{definition}{Definition}[section]

\newtheorem*{remark*}{Remark}

\newtheorem*{ackname}{Acknowledgment}
\newtheorem{step}{Step}
\newtheorem{stepA}{Step}

\theoremstyle{theorem}
\newtheorem{theorem}[definition]{Theorem}
\newtheorem*{theorem*}{Theorem}
\newtheorem{corollary}[definition]{Corollary}
\newtheorem{lemma}[definition]{Lemma}

\newtheorem{claim}[definition]{Claim}

\renewcommand{\phi}{\varphi}
\renewcommand{\Re}{\mathrm{Re}\hspace{1pt}}

\renewcommand{\epsilon}{\varepsilon}

\newcommand{\FOR}{\quad \textrm{for} \quad}
\DeclareMathOperator{\supp}{supp}
\DeclareMathOperator{\meas}{meas}
\DeclareMathOperator{\LF}{\mathcal{L}}

\DeclareMathOperator{\abs}{abs}
\begin{document}

\title[Effective uniform approximation]{Effective uniform approximation by $L$-functions in the Selberg class}

\author[K.~Endo]{Kenta Endo}



\maketitle

\begin{abstract}
Recently, Garunk\v{s}tis, Laurin\u{c}ikas, Matsumoto, J.~\& R.~Steuding showed an effective universality-type theorem for the Riemann zeta-function by using an effective multi-dimensional denseness result of Voronin. 
We will generalize Voronin's effective result and their theorem to the elements of the Selberg class satisfying some conditions. 
\end{abstract}


\section{Introduction and Main results}
Let $\zeta(s)$ denote the Riemann zeta-function, where $s = \sigma + it$ stands for a complex variable. 
In 1975, Voronin \cite{V1975} discovered the universality theorem for the Riemann zeta-function, 
which states as follows;

\begin{theorem}\label{thm:VU}
Let $0 < r < 1/4$. 
Suppose that $f(s)$ is a non-vanishing continuous function on the closed disk $|s| \leq r$ and analytic in the disk $|s| < r$.
Then we have, for any $\epsilon>0$, 
\begin{equation}\label{eqn:LBU}
\liminf_{T \rightarrow \infty} 
\frac{1}{T} \meas \left\{ \tau \in [0 , T] ~;~ \max_{|s| \leq r }\left| \zeta(s + 3/4 + i \tau) - f(s) \right| < \epsilon \right\}
>0.
\end{equation}
\end{theorem}
Roughly speaking, 
Theorem \ref{thm:VU} says that the Riemann zeta-function can approximate any non-vanishing holomorphic function. 
The theory of the universality theorem has been developed in various directions.
See e.g. \cite{L1996}, \cite{S2007} for the general theory of the universality and  a survey paper \cite{M2015} for the recent studies.

In one of these developments, 
there are some studies of the refinement to an effective form of Voronin's universality theorem.
For example, there arose the questions on how large the value \eqref{eqn:LBU} is or how small the shift $\tau$ we can take to make $\zeta(s + i \tau)$ approximate a given non-vanishing holomorphic function. 
In proving the universality theorem, Pe\v{c}erski\u{\i}'s rearrangement theorem \cite{P1973} in Hilbert space and Kronecker's approximation theorem \cite[Appendix 8, Theorem 1]{KV1992} are used.
Since these theorems are ineffective, 
it is difficult in general to obtain the effective results of the universality theorem.

As recent effective results, 
we give examples like \cite{G1981}, \cite{G2003}, \cite{GG2003}, \cite{S2003}, \cite{S2005}, \cite{GLMSS2010}, \cite{LLR2018} and refer to \cite{L2013} for a good survey of effectivization problem of the universality theorem.

Recently, Garunk\v{s}tis, Laurin\u{c}ikas, Matsumoto, J.~\& R.~Steuding \cite{GLMSS2010} refined Matsumoto's universality-type theorem into the effective form. 
Matsumoto's theorem is written in the survey paper \cite[Section $3$]{M2006}.
His idea is that one can deduce universality-type theorem from Voronin's multi-dimensional denseness theorem by using the Taylor expansion series. 
Voronin's multi-dimensional theorem \cite{V1972} asserts the following;  

\begin{theorem}\label{thm:VDT}
Let $n \in \mathbb{N}$ and  $1/2 < \sigma \leq 1$.
Then the set 
\[
\{ ( \zeta(\sigma + i t ), \zeta'(\sigma + i t), \ldots, \zeta^{(n -1)}(\sigma + i t) ) ~;~ t \in \mathbb{R}\}
\]
is dense in $\mathbb{C}^n$.
\end{theorem}

Theorem \ref{thm:VDT} is not effective since Kronecker's and Pe\v{c}erski\u{\i}'s theorems are also used in the proof. 
Hence Matsumoto's universality-type theorem was ineffective. 
To obtain the effective version of Matsumoto's theorem, 
Garunk\v{s}tis, Laurin\u{c}ikas, Matsumoto, J.~\& R.~Steuding directed their attention to the following theorem by Voronin \cite{V1988}, 
which is an effective version of Theorem \ref{thm:VDT}.
\begin{theorem}\label{thm:VET}
Let $N\in\mathbb{N}$ and $\sigma_0 \in (1/2, 1)$ and $\underline{b} = (b_0, b_1, \ldots, b_{N-1}) \in \mathbb{C}^N$ with $| b_0 | > \varepsilon > 0$.
Then a sufficient condition for the system of inequalities 
\[
\left|  \zeta^{(k)}( \sigma_0 + i t) - b_k \right| < \varepsilon, \quad k = 0, \dots, N-1
\]
to have a solution $t \in [T, 2T]$ is that
\[
T > c_0(N, \sigma_0) \exp\exp\left( c_1(N, \sigma_0) A(N, \underline{b}, \epsilon)^{\frac{8}{1 - \sigma_0} + \frac{8}{\sigma_0 - 1/2}}  \right),
\]
where $c_0(N, \sigma_0)$ and $c_1(N, \sigma_0)$ are a positive \footnote{
It is written in \cite{KV1992} that the constant $c_1(N, \sigma_0)$ can be taken $5$.
However one cannot prove this fact by the method in \cite{KV1992}.
This is probably a mistake.
}, effectively computable constant, depending only on $N$ and $\sigma_0$, 
and
\[
A(n, \underline{b}, \epsilon) 
= \left|  \log b_0  \right| + \left( \frac{\| \underline{b} \|}{\epsilon} \right)^{N^2}
\]
with $\| \underline{b} \| = \sum_{0 \leq k <N} | b_k | $. 
Here the above branch of $\log b_0$ can be taken arbitrarily.
\end{theorem}
We remark that this result is also written in the textbook \cite{KV1992}, and the above statement is the form described in the textbook.
This result is also regarded as a kind of $\Omega$-results, which Voronin called it. 
He proved this effective result cleverly without using ineffective results as mentioned above. 
In his proof, 
Pe\v{c}erski\u{\i}'s theorem is replaced by a geometrical argument and an argument in which the system of the linear equation and the prime number theorem for short interval are used. 
Kronecker's approximation theorem is replaced by a kind of amplification technique in which we estimate a certain weighted mean value.
Garunk\v{s}tis, Laurin\u{c}ikas, Matsumoto, J.~\& R.~Steuding \cite{GLMSS2010} used this effective result to refine Matsumoto's theorem. 
Since the author found a slight mistake in their statement \cite{GLMSS2010}, 
we shall mention the modified version of this statement as follows; 

\begin{theorem}[Modified version of the result \cite{GLMSS2010}]
Let $s_0 = \sigma_0 + i t_0$, $1/2 < \sigma_0 < 1$, $r >0$, $\mathcal{K} = \{ s \in \mathbb{C} ~;~ | s - s_0 | \leq r \}$, and suppose that $g:\mathcal{K} \rightarrow \mathbb{C}$ is an analytic function with $g(s_0) \neq 0$. 
Put $M(g) =\max_{ | s - s_0 | = r } | g (s) |$.
Fix $\epsilon \in ( 0, 1)$ and $0 < \delta_0 <1$.
If $N = N ( \delta_0, \epsilon, g)$ and $T = T(g, \epsilon, \sigma_0, \delta_0, N)$ satisfy
\[
M(g) \frac{\delta_0^N}{1 - \delta_0} 
< \frac{\epsilon}{3}
\]
and
\[
T 
\geq \max\left\{ c_0( \sigma_0, N ) \exp\exp \left( c_1(\sigma_0, N) A \left( N, \mathbf{g}, (\epsilon/3)\exp( - \delta_0 r ) \right)^{d(\sigma_0, E_{\LF})} \right) , r \right\}, 
\]
respectively, 
then there exists $\tau \in [T -  t_0, 2T - t_0 ]$ such that
\[
\max_{| s - s_0 | \leq \delta r} \left| \zeta ( s + i \tau) - g (s) \right| < \epsilon
\]
for any $0 \leq \delta \leq \delta_0$ satisfying
\[
M(\tau) \frac{\delta^N}{1 - \delta}
< \frac{\epsilon}{3}.
\]
Here $c_0 ( \sigma_0, N )$, $c_1 ( \sigma_0, N )$ and $A \left( N, \mathbf{g}, (\epsilon/3)\exp( - \delta_0 r ) \right)$ are the same constants as in Theorem \ref{thm:VET} with $\mathbf{g}
= \left( g(s_0), g'(s_0), \ldots, g^{(N-1)} (s_0) \right)$
, and $M(\tau)$ is defined by $M(\tau) = \max_{ | s - s_0 | = r } | \zeta( s + i \tau ) |$.
\end{theorem}
This correction is inspired by Matsumoto's paper \cite{M2006}.
We will mention this correction in the proof of Corollary \ref{cor:GLMSS}. 
Our goal in this paper is to generalize these results in \cite{V1988} and \cite{GLMSS2010} to the element of the Selberg class $\mathcal{S}$ satisfying some conditions.

To state the main theorem, 
we start to recall the definition of the Selberg class $\mathcal{S}$.
The Dirichlet series 
\[
\mathcal{L}(s) = \sum_{n = 1}^{\infty}\frac{a(n)}{n^s}
\]
is said to belong to the Selberg class $\mathcal{S}$ if $\mathcal{L}(s)$ satisfies the following axioms; 
\begin{enumerate}[(i)]
\item Ramanujan hypothesis: $a(n) \ll_\varepsilon n^\varepsilon$ for any $\varepsilon>0$. 
\item Analytic continuation: there exists a nonnegative integer $m$ such that $(s -1)^m \mathcal{L}(s)$ is an entire function of finite order. 
\item Functional equation: $\mathcal{L}(s)$ satisfies a functional equation of the type 
\[
\mathcal{H}_\mathcal{L}(s)
= \omega \overline{\mathcal{H}_\mathcal{L}(1 - \overline{s})}
\]
where
\[
\mathcal{H}_\mathcal{L}(s) = \mathcal{L}(s) R^s \prod_{j=1}^{f}\Gamma(\lambda_j s + \mu_j) = \gamma (s) \LF(s)
\]
with positive real numbers $R$, $\lambda_j$ and complex numbers $\mu_j$ and $\omega$ with  $\Re \mu_j \geq 0$ and $|\omega| = 1$. 
\item Euler product: $\log \LF (s) = \sum_{n =1}^{\infty} b(n) n^{- s}$, 
where $b(n)= 0$ unless $n = p^m$ with a prime number $p$ and $m \geq 1$, and $b(n) \ll n^{\vartheta}$ for some $\vartheta < 1/2$.
\end{enumerate}

Now we give some definitions. 
The zeros of $\LF(s)$ which are not derived from the poles of the $\gamma$-factor $\gamma(s)$ and are not equal to possible zero of $\LF(s)$ at $s=0, 1$ are called non-trivial zeros, and denote by $\rho = \beta + i \gamma$ such zeros throughout this paper.
Let $N_{\LF}(T)$ denote the number of non-trivial zeros with multiplicity satisfying $0 \leq \beta \leq 1$ and $|\gamma| \leq T$. 
We remark that it is known that
\begin{equation}\label{eqn:GRV}
N_{\LF}(T) = \frac{d_{\LF}}{\pi} T \log T + c_{\LF} T + O(\log T)
\end{equation}
holds, 
where $d_{\LF} = 2 \sum_{j =1}^{f} \lambda_j$ and $c_{\LF}$ is some constant depending on $\LF$.
This $d_{\LF}$ is called the degree of $\LF(s)$ and known to be invariant in the Selberg class $\mathcal{S}$.
For other properties of the Selberg class $\mathcal{S}$, 
we refer to a survey paper \cite{P2005} for example.

In this paper, 
we further assume the following three conditions; 
\begin{itemize}
\item[(C1)] There exists a $\kappa = \kappa (\LF) >0$ such that
\[
\frac{1}{\pi (X)}\sum_{p \leq X}\left| a(p) \right|^2 
\sim \kappa  \quad \textrm{as $X \rightarrow \infty$}. 
\]
\item[(C2)] There exists a $\sigma_{\LF} \geq 1/2$ such that for any fixed $\sigma> \sigma_{\LF}$ 
\[
N_{\LF}( \sigma, T ) 
\ll T^{1 - \Delta_{\LF} (\sigma)}
\]
as $T \rightarrow \infty$ with some positive real number $\Delta_{\LF}(\sigma) > 0$, 
where $N_{\LF}( \sigma, T )$ denote the number of non-trivial zeros of $\LF(s)$ with multiplicity satisfying $\beta \geq \sigma$ and $|\gamma| \leq T$.
The implicit constant may depend on $\sigma$.
\item[(C3)] There exists an $E_{\LF}>0$ such that 
\[
\sum_{X < p \leq X + H}\left| a(p) \right|^2 
\sim \kappa \frac{ H }{\log X} \quad 
\textrm{and} \quad \pi(X + H) - \pi(X) 
\sim \frac{ H }{\log X}
\]
hold for $X \geq H \geq X^{1 - E_{\LF}} (\log X)^D$ with some $D \geq 1$ as $X \rightarrow \infty$.
\end{itemize}

The above implicit constants appearing in the symbol $O(\cdot)$ and $\ll$ may depend on $\LF(s)$. 
Remark that the universality theorem for the element of the Selberg class $\mathcal{S}$ satisfying the condition {\rm (C1)} is proved in a certain strip by Nagoshi and Steuding \cite{NS2010}.

Here, we define a branch of $\log \LF (s)$. 
Let $G(\LF)$ denote
\[
G(\LF) =
\left\{s ~;~ \sigma > 1/2 \right\} \setminus \left\{ \left( \bigcup_{\rho = \beta + i \gamma} \left\{ s = \sigma + i \gamma ~;~ \sigma \leq \beta \right\}  \right) \cup  (- \infty, 1]  \right\}
\]
and we define $\log \LF (s)$ by 
\[
\log \LF (s)
= \int_{\infty}^{\sigma} \frac{\LF'}{\LF} ( \alpha + it) d\alpha
\]
for $ s = \sigma + it \in G(\LF) $.

In this paper, we first show the following two results which are generalizations of the results in \cite{V1988}.

\begin{theorem}\label{thm:MTH}
Let $\mathcal{L}(s) = \sum_{n =1}^\infty a(n) n^{-s}$ be an element of the Selberg class $\mathcal{S}$ satisfying the conditions {\rm (C1)}, {\rm (C2)} and {\rm (C3)}.
Let $\max\{ \sigma_{\LF} , 1 - 2 E_{\LF}  \} < \sigma_0 < 1$, $N \in \mathbb{N}$, $\varepsilon \in (0,1)$, $\underline{c} = (c_0, c_1, \ldots, c_{N -1}) \in \mathbb{C}^N$.
Then a sufficient condition for the system of inequalities 
\[
\left| \frac{d^k}{d s^k} \log \LF (\sigma_0 + it) - c_k \right| < \varepsilon \quad \textrm{for $k = 0, 1, \ldots, N -1$}
\]
to have a solution $t \in [T, 2T]$ is that
\[
T \geq 
\exp \exp \left(C_1 (\LF, \sigma_0, N ) \left( \| \underline{c} \| + \frac{1}{\varepsilon}  \right)^{d(\sigma_0, E_{\LF})} \right),
\]
where $C_1 (\LF, \sigma_0, N )$ and $d(\sigma_0, E_{\LF})$ are effectively computable positive constants depending on $\LF, \sigma_0, N$ and on $\sigma_0, E_{\LF}$ respectively.
\end{theorem}

\begin{corollary}\label{cor:CV}
Under the same hypothesis of Theorem \ref{thm:MTH} with $c_0 \neq 0$,
a sufficient condition for the system of inequalities
\[
\left| \frac{d^k}{d s^k} \LF (\sigma_0 + it) - c_k \right| < \varepsilon \quad \textrm{for $k = 0, 1, \ldots, N -1$}
\]
to have a solution $t \in [T,2T]$ is that
\[
T 
\geq \exp \exp \left( C_2 ( \LF, \sigma_0, N ) B(N, \underline{c}, \epsilon) ^{d(\sigma_0, E_{\LF})} \right),
\]
where $ C_2 ( \LF, \sigma_0, N ) $ is effectively computable positive constant depending on $\LF, \sigma_0, N$, 
and $d(\sigma_0, E_{\LF})$ is the same constant as in Theorem \ref{thm:MTH}, and
\[
B(N, \underline{c}, \epsilon)
= | \log c_0 | 
+ \left( \frac{ \| \underline{c} \| }{ | c_0 | } \right)^{(N-1)^2} \frac{1 + | c_0 |}{ \epsilon}.
\]
Here the above branch of $\log c_0$ can be taken arbitrarily.
\end{corollary}

By combining Corollary \ref{cor:CV} with the method as in \cite{GLMSS2010}, 
we have the following corollary.

\begin{corollary}\label{cor:GLMSS}
Let $\mathcal{L}(s) = \sum_{n =1}^\infty a(n) n^{-s}$ be an element of the Selberg class $\mathcal{S}$ satisfying the conditions {\rm (C1)}, {\rm (C2)} and {\rm (C3)}.
Let $s_0 = \sigma_0 + i t_0$, $\max\{ \sigma_{\LF} , 1 - 2 E_{\LF}  \} < \sigma_0 < 1$, $r >0$, $\mathcal{K} = \{ s \in \mathbb{C} ~;~ | s - s_0 | \leq r \}$, and suppose that $g:\mathcal{K} \rightarrow \mathbb{C}$ is an analytic function with $g(s_0) \neq 0$. 
Put $M(g) =\max_{ | s - s_0 | = r } | g (s) |$.
Fix $\epsilon \in ( 0, 1)$ and $0 < \delta_0 <1$.
If $N = N ( \delta_0, \epsilon, M(g))$ and $T = T(\LF, g, \epsilon, \sigma_0, \delta_0, N)$ satisfy
\[
M(g) \frac{\delta_0^N}{1 - \delta_0} 
< \frac{\epsilon}{3}
\]
and
\[
T 
\geq \max\left\{ \exp\exp \left( C_2(\LF, \sigma_0, N) B \left( N, \mathbf{g}, (\epsilon/3)\exp( - \delta_0 r ) \right)^{d(\sigma_0, E_{\LF})} \right) , r \right\}, 
\]
respectively, 
then there exists $\tau \in [T -  t_0, 2T - t_0 ]$ such that
\[
\max_{| s - s_0 | \leq \delta r} \left| \LF ( s + i \tau) - g (s) \right| < \epsilon
\]
for any $0 \leq \delta \leq \delta_0$ satisfying
\[
M(\tau;\LF) \frac{\delta^N}{1 - \delta}
< \frac{\epsilon}{3}.
\]
Here $ C_2 ( \LF, \sigma_0, N ) $ and $B \left( N, \mathbf{g}, (\epsilon/3)\exp( - \delta_0 r ) \right)$ are the same constants as in Corollary \ref{cor:CV} with $\mathbf{g}
= \left( g(s_0), g'(s_0), \ldots, g^{(N-1)} (s_0) \right)$, 
and $M(\tau;\LF)$ is defined by $M(\tau;\LF) = \max_{ | s - s_0 | = r } | \LF( s + i \tau ) |$.
\end{corollary}

At the end of this section, we give some examples.
First, we will see that one can apply these results to the Riemann zeta-function in the range $1/2 < \sigma_0 <1$.
In this case, the condition (C1) is well-known as the prime number theorem.
For the condition (C2), we can take $\sigma_{\zeta} = 1/2$ by the zero-density theorem (see e.g. \cite[Theorem 9.19]{T1986}).
For the condition (C3), it is known that
\[
\pi(X + H) - \pi(X)
\sim \frac{H}{\log X} 
\FOR X^{7/12} (\log X)^{22} \leq H \leq X
\]
holds (see e.g. \cite[Theorem 12.8]{I1985}). 
Hence we can take $E_{\zeta} = 5/12$ and we obtain $\max\{\sigma_{\zeta}, 1 - 2 E_\zeta\} = 1/2$ in this case.
For other examples, 
we can confirm that the Dirichlet $L$-function $L(s, \chi)$ of primitive characters $\chi$ and the Dedekind zeta-function $\zeta_K(s)$ belong to the Selberg class $\mathcal{S}$ and satisfy the conditions (C1), (C2) and (C3).

\section{Preliminaries}
In this section, 
let $\mathcal{L}(s)$ be an element of the Selberg class $\mathcal{S}$ which is represented by
\[
\mathcal{L}(s) = \sum_{n = 1}^{\infty} \frac{a(n)}{ n^{s}} = \prod_p \exp\left( \sum_{k =1}^{\infty}\frac{b(p^k)}{p^{ks}} \right) \quad \textrm{for} \quad \sigma >1.
\] 

\subsection{Definitions and Notations}
We use the following definitions and notations.
\begin{itemize}
\item Let $\mathbb{N}_0$ denote the set of nonnegative integers.
\item Let $\mathcal{P}$ denote the set of all prime numbers.
\item For a set $A$, let $\mathbb{R}^A$ denote the family of elements in $\mathbb{R}$ indexed by $A$. 
\item For any $Q>0$, 
let $\mathcal{P}(Q)$ denote the set of prime numbers smaller than or equal to $Q$.
\item We define the generalized von Mangoldt function $\Lambda_{\mathcal{L}}(n)$ by
\[
- \frac{\LF'}{\LF}(s) = \sum_{n=1}^\infty \Lambda_{\LF}(n)n^{-s}
\]
for $\sigma >1$, that is, $\Lambda_{\LF}(n)= b(n) \log n.$
\item Let $\mathcal{M}$ be a finite subset of $\mathcal{P}$ and $s = \sigma + it$ be a complex number with $\sigma> 1/2$. 
For any $\mathcal{M} \subset \mathcal{N} \subset \mathcal{P}$ and $\underline{\theta} = ( \theta_p )_{p \in \mathcal{N}} \in \mathbb{R}^{\mathcal{N}}$, we denote  $\phi_{\mathcal{M}} (s, \underline{\theta})$ and $\log \LF_{\mathcal{M}} (s, \underline{\theta})$ by 
\[
\phi_{\mathcal{M}} (s, \underline{\theta})
 = \sum_{p \in \mathcal{M}}\frac{b(p) \exp (- 2 \pi i \theta_p)}{p^s}
\]
and
\[
\log \LF_{\mathcal{M}} (s, \underline{\theta})
= \sum_{p \in \mathcal{M}} \sum_{l =1}^{\infty} \frac{b(p^l) \exp( - 2 \pi i l \theta_p ) }{p^{l s}}.
\]
Note that the series $\log \LF_{\mathcal{M}} (s, \underline{\theta})$ converges absolutely by the estimate $b(p^l) \ll p^{l \vartheta}$ with some $\vartheta < 1/2$ coming from the axiom (iv) of the Selberg class.
If $\mathcal{M}=\{ p \}$, 
we abbreviate $\phi_{\{p\}} (s, \underline{\theta})$ and $\log \LF_{\{p\}} (s, \underline{\theta})$ to $\phi_{p} (s, \underline{\theta})$ and $\log \LF_{p} (s, \underline{\theta})$ respectively, 
and if $\underline{\theta} = (0)_{p \in \mathcal{N}}$, we do $\log \LF_{\mathcal{M}} (s, (0)_{p \in \mathcal{N} })$ to $\log \LF_{\mathcal{M}} (s)$.
\end{itemize}
\subsection{Some known results}
In this subsection, 
we summarize the results not coming from analytic number theory. 
The followings are used in \cite{KV1992} and \cite{V1988}.

\subsubsection{A certain estimate coming from the Vandermonde matrix}
\begin{lemma}\label{lem:Van}
Let $X > e$, $N \in \mathbb{N}$ and $\underline{a} = (a_0, a_1, \ldots , a_{N -1}) \in \mathbb{C}^N$.
Put $X_j = 2^{j} X$ for $j = 0, 1, \ldots, N -1$.
Then the system of linear equations in the unknown $\underline{z} = (z_0, z_1, \ldots , z_{N -1}) \in \mathbb{C}^N$;
\[
\begin{pmatrix}
1 & 1 & \cdots & 1 \\
- \log X_0 & - \log X_1 & \ldots &  - \log X_{N-1} \\
 \left(- \log X_0\right)^2 & \left( - \log  X_1\right)^2 & \ldots & \left(- \log X_{N-1}\right)^2 \\
\vdots & \vdots & \ddots & \vdots \\
\left(-\log X_0 \right)^{N -1} & \left(- \log X_1\right)^{N -1} & \cdots & \left(- \log X_{N-1} \right)^{N-1} 
\end{pmatrix}
\begin{pmatrix}
z_0  \\
z_1  \\
z_2  \\
\vdots  \\
z_{N-1}
\end{pmatrix}
= 
\begin{pmatrix}
a_0  \\
a_1  \\
a_2  \\
\vdots  \\
a_{N-1}
\end{pmatrix}
\]
has a unique solution $\underline{z}_0 = \underline{z}_0(X,  \underline{a} ) \in \mathbb{C}^{N}$ and the estimate
\[
\|\underline{z}_0\| 
\ll_N ( \log X )^{N -1} \| \underline{a} \|
\]
holds.
\end{lemma}

This lemma is used in \cite{KV1992} and \cite{V1988}, 
however the proof was written very roughly in these papers.
For this reason, we will give a proof in this paper.
\begin{proof}
Fix $ \underline{a} = (a_0, a_1, \ldots , a_{N -1}) \in \mathbb{C}^N$.
Let $\mathbf{U} = (U_0, U_1, \ldots, U_{N -1})$ be indeterminate and $\underline{z} = (z_0, z_1, \ldots , z_{N -1})$ variable in $\mathbb{C}(\mathbf{U})^N$, 
where $\mathbb{C}(\mathbf{U})$ denotes the field of rational functions.
We first consider the following system of linear equations;
\begin{equation}\label{eqn:Van}
\begin{pmatrix}
1 & 1 & \cdots & 1 \\
U_0 & U_1 & \ldots &  U_{N-1} \\
 U_0^2 & U_1^2 & \ldots & U_{N-1} ^2 \\
\vdots & \vdots & \ddots & \vdots \\
U_0^{N -1} & U_1^{N -1} & \cdots & U_{N-1}^{N-1} 
\end{pmatrix}
\begin{pmatrix}
z_0  \\
z_1  \\
z_2  \\
\vdots  \\
z_{N-1}
\end{pmatrix}
= 
\begin{pmatrix}
a_0  \\
a_1  \\
a_2  \\
\vdots  \\
a_{N-1}
\end{pmatrix}.
\end{equation}
Let $B(\mathbf{U})$ denote the first matrix in the left hand side of \eqref{eqn:Van}.
Since $B(\mathbf{U})$ is the Vandermonde matrix, we have
\begin{equation}\label{eqn:Vande}
\det \left( B(\mathbf{U})\right)
= \prod_{0 \leq l < k \leq N -1} (U_k - U_l) \neq 0 \quad \textrm{in $\mathbb{C}[\mathbf{U}]$} .
\end{equation}
Hence the system of linear equations \eqref{eqn:Van} has a solution $\underline{z} = \underline{z} (\mathbf{U}, \underline{a}) \in \mathbb{C}(\mathbf{U})^N$ such that
\[
{}^t\underline{z} = \frac{1}{\det \left( B(\mathbf{U})\right)} \widetilde{B}(\mathbf{U}) \cdot {}^t \underline{a}, 
\]
where $\widetilde{B}(\mathbf{U}) = {}^t( \widetilde{b}_{i,j} (\mathbf{U}) )_{0 \leq i, j \leq N -1}$ denotes the adjugate matrix of $B(\mathbf{U})$.
Here the symbol ${}^t D$ stands for the transposed matrix of a matrix $D$.
We fix $0 \leq i, j \leq N -1$ and put
\[
\Delta_{j} ( \mathbf{U} )
= \prod_{\substack{0 \leq l < k \leq N -1; \\ l, k \neq j}} (U_k - U_l).
\]
Then we find that $\Delta_{j} ( \mathbf{U} )$ divides $\widetilde{b}_{i,j} (\mathbf{U})$, 
and hence there exists $f_{i,j} ( \mathbf{U} ) \in \mathbb{C}[\mathbf{U}]$ such that $\widetilde{b}_{i,j} (\mathbf{U}) = f_{i, j} ( \mathbf{U} )  \Delta_{j} ( \mathbf{U} )$.
By the definition of the determinant, we have
\[
\deg \left( \widetilde{b}_{i,j} (\mathbf{U}) \right)
= \sum_{\beta = 0}^{N -1} \beta - i.
\]
Here the statement $\deg(g (\mathbf{U})) = n$ means that
\[
\max\left\{ i_1 + \cdots + i_{N -1} ~;~ c_{i_1, \ldots, i_{N -1}} \neq 0\right\} = n
\]
for $g(\mathbf{U}) = \sum_{i_0 \ldots, i_{N -1}} c_{i_0, \ldots, i_{N-1}} U_0^{i_0}\cdots U_{N -1}^{i_{N -1}} \in \mathbb{C}[\mathbf{U}] $.
On the other hand, we find that
\begin{align*}
\deg \left( \Delta_{j} (\mathbf{U}) \right) 
&= \#\left\{ (l, k) \in \{0, 1, \ldots, N -1\} ~;~ l < k, l,k \neq j\right\} \\
&= \sum_{\beta = 0}^{N -1} \beta - (N -1).
\end{align*}
Therefore we have
\begin{align*}
\deg \left( f_{i, j} (\mathbf{U}) \right)
&= \deg \left( \widetilde{b}_{i,j} (\mathbf{U}) \right) - \deg \left( \Delta_{j} (\mathbf{U}) \right) \\
&= \left( \sum_{\beta = 0}^{N -1} \beta - i \right) - \left( \sum_{\beta = 0}^{N -1} \beta - (N -1) \right) \\
&= N -1 - i 
\leq N -1.
\end{align*}
By substituting $\mathbf{U}$ into $\mathbf{X} = ( - \log X_0, - \log X_1, \ldots, - \log X_{N -1} )$ and letting $\underline{z}_0 = \underline{z}_0 (X, \underline{a})$ be a solution of the system of linear equations \eqref{eqn:Van} in this case, 
the estimates
\[
\left| \widetilde{b}_{i,j} (\mathbf{X}) \right|
= \left| \Delta_{j} (\mathbf{X}) \right| \left|  f_{i, j} (\mathbf{X}) \right|
\ll_N ( \log X )^{N -1} 
\quad \textrm{and} \quad 
\det \left( B(\mathbf{X})\right) \asymp_{N} 1
\]
hold by the equation \eqref{eqn:Vande} and the definition of $\Delta_{j} (\mathbf{X})$.
Consequently, we obtain
\[
\|  \underline{z}_0 \|
\ll_N \left( \sum_{0 \leq i, j \leq N -1} \left| \widetilde{b}_{i,j} (\mathbf{X}) \right| \right) \| \underline{a} \|
\ll_N ( \log X )^{N -1} \| \underline{a} \|.
\]
This completes the proof.
\end{proof}

\subsubsection{Elementary geometric lemma}
\begin{lemma}\label{lem:EG}
Let $N$ be a positive integer larger than two.
Suppose that the positive numbers $r_1 \leq r_2 \leq  \cdots \leq r_N$ satisfy $r_N \leq \sum_{n =1}^{N - 1} r_n$.
Then we have
\begin{align*}	
\left\{ \sum_{n =1}^{N}  r_n \exp ( - 2 \pi i \theta_n ) \in \mathbb{C} ~;~ \theta_n \in [0,1) \right\}
=\left\{ z \in \mathbb{C} ~;~ | z | \leq \sum_{n =1}^N r_n \right\}.
\end{align*}
\end{lemma}
\begin{proof}
The proof is written in \cite{CG2014} roughly and in \cite{T2013} precisely.
\end{proof}

\subsubsection{The mollifier and the estimate for its Fourier series expansion}\label{sbs:MOF}
Let $Q, M>2$ and put $\delta= \delta_Q = Q^{-1} \in (0,1)$.
We will prepare a mollifier on $\mathbb{R}^{\mathcal{P}(Q)}$ and its truncated formula as follows;
We take $\phi \in C^{\infty}(\mathbb{R})$ satisfying
\[
\phi (x) \geq 0,  \quad
\supp (\phi) \subset [ -1, 1], \quad
\int^\infty_{- \infty} \phi (x) dx = 1.
\]
Throughout this subsection,
let the implicit constants depend on $\phi$.
We define the function $\phi_\delta : [ - 1/2, 1/2 ] \rightarrow \mathbb{R}$ by
\[
\phi_\delta ( \theta ) = \frac{1}{\delta} \phi \left( \frac{\theta}{\delta} \right), \quad
\theta \in \left[ - 1/2 , 1/2 \right] 
\]
and extend $\phi_\delta$ onto $\mathbb{R}$ by periodicity with period $1$.
We also define $\Phi_Q (\underline{\theta}): \mathbb{R}^{\mathcal{P}(Q)} \rightarrow \mathbb{R}$ by 
\[
\Phi_Q (\underline{\theta}) = \prod_{p \leq Q} \phi_\delta(\theta_p), \quad
\underline{\theta} = (\theta_p)_{p \in \mathcal{P}(Q)} \in \mathbb{R}^{\mathcal{P}(Q)}.
\]
Note that
\begin{equation}\label{eqn:SUPP}
\Phi_Q (\underline{\theta}) \neq 0 \quad \textrm{implies} \quad \underline{\theta} \in [-\delta, \delta]^{\mathcal{P}(Q)} + \mathbb{Z}^{\mathcal{P}(Q)}.
\end{equation}
For any $\theta_0 \in \mathbb{R}$, 
the function $ \phi_{\delta} ( \theta - \theta_0 )$, $\theta \in \mathbb{R}$ can be represented as a Fourier series
\[
\phi_\delta(\theta - \theta_0) 
= \sum_{n \in \mathbb{Z}} \alpha_n ( \theta_0 ) \exp( 2\pi i n \theta ) 
\]
where the Fourier coefficients are given by
\[
\alpha_n (\theta_0)
= \int_{-1/2}^{1/2} \phi_\delta ( \theta - \theta_0 ) \exp( -2 \pi i n \theta ) d \theta.
\]
By integration by parts, we have
\begin{equation}\label{eqn:AL}
\alpha_0 (\theta_0) = 1 
\quad \textrm{and} \quad
|\alpha_n ( \theta_0 )| \leq \min\{ 1, C_\phi / \delta^2 n^2 \}
\end{equation}
with some positive constant $C_\phi$ depending on $\phi.$
Then we have
\begin{align*}
&\Phi_Q (\underline{\theta} - \underline{\theta}^{(0)})
=\sum_{\underline{n} \in \mathbb{Z}^{\mathcal{P}(Q)} }\beta_{\underline{n}}(\underline{\theta}^{(0)}) \exp \left( 2 \pi i \langle \underline{n}, \underline{\theta} \rangle \right)  \\
=&\sum_{ \substack{ \underline{n} = (n_p)_{p \in \mathcal{P}(Q)} \in \mathbb{Z}^{\mathcal{P}(Q)} ; \\\max_{p \leq Q} |n_p| \leq M }} 
\beta_{\underline{n}}(\underline{\theta}^{(0)}) \exp \left( 2 \pi i \langle \underline{n}, \underline{\theta} \rangle \right)
+ O\left( \sum_{ \substack{ \underline{n} \in \mathbb{Z}^{\mathcal{P}(Q)} ; \\\max_{p \leq Q} |n_p| > M }} \left| \beta_{\underline{n}}(\underline{\theta}^{(0)}) \right| \right)
\end{align*}
for $\underline{\theta} =  (\theta_p)_{p \in \mathcal{P}(Q)} \in \mathbb{R}^{\mathcal{P}(Q)}$, where
\[
\langle  \underline{n}, \underline{\theta} \rangle 
= \sum_{p \leq Q} n_p \theta_p
\quad \textrm{and} \quad 
\beta_{\underline{n}}(\underline{\theta}^{(0)})
= \prod_{p \leq Q} \alpha_{n_p}(\theta_p^{(0)}).
\]
Note that the estimates 
\[
\beta_{\mathbf{0}}(\underline{\theta}^{(0)})=1
\quad \textrm{and} \quad
|\beta_{\underline{n}} (\underline{\theta}^{(0)})|
\leq \prod_{p \leq Q} \min\{ 1, C_\phi / \delta^2 n_p^2 \}
\]
hold
by the estimate \eqref{eqn:AL}.
By using the prime number theorem and by noting
\[
\left\{
\underline{n} = (n_p)_{p \in \mathcal{P}(Q) } \in \mathbb{Z}^{\mathcal{P}(Q)} ~ ; ~ \max_{p \leq Q} | n_p | > M
\right\}
=
\bigcup_{q \in \mathcal{P}(Q)}\left\{ \underline{n} \in \mathbb{Z}^{\mathcal{P}(Q)} ~;~ | n_q | > M \right\},
\]
we have
\begin{equation}\label{eqn:EB}
\sum_{ \underline{n} \in \mathbb{Z}^{\mathcal{P}(Q)}} | \beta_{\underline{n}} (\underline{\theta}^{(0)}) |
\leq \left( \sum_{n \in \mathbb{Z}} \min\left\{ 1, \frac{C_\phi}{\delta^2 n^2} \right\} \right)^{\pi (Q)}
\ll \exp ( C_0 Q )
\end{equation}
and 
\begin{align*}
&\sum_{ \substack{ \underline{n} = (n_p)_{p \in \mathcal{P}(Q)} \in \mathbb{Z}^{\mathcal{P}(Q)} ; \\\max_{p \leq Q} |n_p| > M }} \left| \beta_{\underline{n}}(\underline{\theta}^{(0)}) \right| \\
\leq& \pi(Q) \left(\sum_{\substack{n \in \mathbb{Z};\\ |n|>M }} \frac{1}{\delta^2 n^2} \right) \left( \sum_{n \in \mathbb{Z}} \min\left\{ 1, \frac{1}{\delta^2 n^2} \right\} \right)^{\pi(Q) - 1}
\ll \frac{1}{M} \exp( C_1 Q )
\end{align*}
for some positive constant $C_0, C_1$ depending on $\phi$.
Hence we have
\begin{align}
&\Phi_Q (\underline{\theta} - \underline{\theta}^{(0)}) \label{eqn:MOL}
\\
=&\sum_{ \substack{ \underline{n} = (n_p)_{p \in \mathcal{P}(Q)} \in \mathbb{Z}^{\mathcal{P}(Q)} ; \\\max_{p \leq Q} |n_p| \leq M }} 
\beta_{\underline{n}}(\underline{\theta}^{(0)}) \exp \left( 2 \pi i \langle \underline{n}, \underline{\theta} \rangle \right)
+ O \left( \frac{1}{M} \exp( C_1 Q ) \right) . \nonumber
\end{align}

\subsection{Known results for the Selberg class $\mathcal{S}$}
\begin{lemma}\label{lem:CS}
The following statements hold.
\begin{itemize}
\item[(i)] We have $a(p) = b(p)$ for all primes $p$.
\item[(ii)] For any $\varepsilon>0$, we have the inequality
\[
| b(p^l) | \ll_\varepsilon 
(2^l -1 ) p^{l\varepsilon}/l
\]
for all primes $p$ and all $l \in\mathbb{N}$, 
where the implicit constant may depend on $\LF(s)$.
\end{itemize}
\end{lemma}
\begin{proof}
The proof can be found in \cite[Exercise 8.2.9]{M2007}.
\end{proof}

\begin{lemma}\label{thm:SF}
Let $q_j(r) = (\mu_j + r )/\lambda_j$ for $j =1,\ldots f$ and $r \in \mathbb{N}_0$.
If $x>1$ and $s \neq 0, 1, - q_j(r), \rho$ for $j =1,\ldots f$ and $r \in \mathbb{N}_0$, then we have
\begin{align}
\frac{\LF'}{\LF}(s)
&= - \sum_{n \leq x^2}\frac{\Lambda_{\LF,x}(n)}{n^s}
+ \frac{1}{\log x}\sum_{j =1}^f \sum_{r =0}^\infty \frac{ x^{- q_j(r) -s } - x^{ - 2 (  q_j(r)  + s ) } }{( q_j(r) +s )^2}\label{eqn:SF} \\
&\quad + m_{\LF} \frac{x^{-2s} - x^{-s}}{s^2 \log x} 
+ m_{\LF} \frac{x^{2(1-s)} - x^{1-s} }{(1-s)^2 \log x}
+\frac{1}{\log x}\sum_{\rho}\frac{x^{\rho - s} - x^{ 2 ( \rho - s ) }}{(\rho - s)^2},\nonumber
\end{align}
where $\Lambda_{\LF,x}(n)$ is defined by
\[
\Lambda_{\LF,x}(n) = \Lambda_{\LF}(n) \quad \textrm{for $1 \leq n \leq x$}, \quad
\Lambda_{\LF}(n) \frac{\log (x^2 /n)}{\log x} \quad \textrm{for $x\leq n \leq x^2$},
\]
and $m_{\LF}$ is defined by
\[
m_{\LF}
=\begin{cases}
    \textrm{the order of pole of $\mathcal{L}(s)$ at $s=1$} \quad  \textrm{if $\LF(s)$ has a pole at $s =1$}, \\
    0 \quad  \textrm{if $\LF(s)$ has no zeros or poles at $s =1$},\\ 
    (-1) \times \textrm{$($the order of zero of $\mathcal{L}(s)$ at $s =1$$)$} \quad \textrm{if $\LF(1) = 0$}.
  \end{cases}
\]
\end{lemma}
\begin{proof}
This follows by the same argument as in \cite[Theorem 14.20]{T1986}.
\end{proof}

\section{Proofs}
\subsection{Proof of Theorem \ref{thm:MTH}}
We fix $\max\{ \sigma_{\LF} , 1 - 2 E_{\LF}  \} < \sigma_0 < 1$
and $N \in \mathbb{N}$, $\underline{c} = (c_k)_{k =0}^{N-1} \in \mathbb{C}^N$ and take $\epsilon \in (0, 1)$.
We begin with the following lemma.
\begin{lemma}\label{lem:AP}
There exist $d_1(\sigma_0, E_{\LF})>0$ and $C (\LF,  \sigma_0, N ) > 0$ such that if
\[
Q > C (\LF, \sigma_0, N ) \left( \| \underline{c} \| + 1 / \varepsilon  \right)^{d_1(\sigma_0, E_{\LF})},
\]
then there exist $\underline{\theta}^{(\star)} = \underline{\theta}^{(\star)}(Q) = ( \theta_p^{(\star)} )_{p \in \mathcal{P}(Q)} \in \mathbb{R}^{\mathcal{P}(Q)}$ such that
\[
\left| \frac{d^k}{d s^k} \log \LF_{\mathcal{P}(Q)} (\sigma_0, \underline{\theta}^{(\star)}) - c_k \right| < \frac{\epsilon}{3} \quad \textrm{for $k = 0, 1, \ldots, N -1$}.
\]
Here $ (d^k/ds^k) F (\sigma) $ means $(d^k/ds^k) F (s)|_{s = \sigma}$ for holomorphic $F(s)$.
\end{lemma}

\begin{proof}[Proof of Lemma \ref{lem:AP}]
We divide the proof into several steps.
\begin{step}
We will show the following claim and give a certain convergent series.
\begin{claim}\label{claim:PS}
There exists $\underline{\theta}^{(0)} = ( \theta_{p_n}^{(0)} )_{n=1}^{\infty} \in \mathbb{R}^{\mathbb{N}}$ such that the estimate
\begin{equation}\label{eqn:ALT}
\left| \sum_{p \leq \xi} b (p) \exp ( - 2 \pi i \theta_p^{(0)}) \right|
 \leq C_{\LF, \eta} \xi^{\eta}
\end{equation}
holds for any $\xi > 0$ when the estimate $ | b ( p ) | = |a (p)| \leq C_{\LF, \eta} p^\eta $ holds for any prime number $p$ with some positive number $\eta$.
\end{claim}
\begin{proof}[Proof of Claim \ref{claim:PS}]
We put $\mathcal{P}_0 = \{ p \in \mathcal{P} ~;~ b(p) \neq 0 \}$ and $\{ p_n \}_{n=1}^\infty$ which satisfy $\{ p_n \}_{n =1}^{\infty} = \mathcal{P}_0$ and $p_n < p_{n +1}$ for any $n \in \mathbb{N}$.
For any $n\geq 1$, 
put
\[
b(p_n)
=| b(p_n) | \exp ( 2 \pi i \theta_{p_n}^{(\mathcal{L})} ), 
\]
and we take $\underline{\theta}^{(0)} = ( \theta_{p_n}^{(0)} )_{n=1}^{\infty} \in \mathbb{R}^{\mathbb{N}}$ so that
\[
\theta_{p_1}^{(0)}
= \theta_{p_1}^{(\mathcal{L})}, \quad
\theta_{p_2}^{(0)}
= 1/2 + \theta_{p_2}^{(\mathcal{L})},
\]
\[
\theta_{p_3}^{(0)}
=\begin{cases}
    \theta_{p_3}^{(\mathcal{L})} \quad  \textrm{if} \quad \displaystyle\sum_{j =1}^{2} b(p_j)\exp(- 2 \pi i \theta_{p_j}^{(0)}) \leq 0, \\
    1/2 + \theta_{p_3}^{(\mathcal{L})} \quad  \textrm{if} \quad \displaystyle\sum_{j =1}^{2} b(p_j)\exp(- 2 \pi i \theta_{p_j}^{(0)}) > 0,
  \end{cases}
\]
\[
\vdots
\]
\[
\theta_{p_l}^{(0)}
=\begin{cases}
    \theta_{p_l}^{(\mathcal{L})} \quad  \textrm{if} \quad \displaystyle\sum_{j =1}^{l -1} b(p_j)\exp(- 2 \pi i \theta_{p_j}^{(0)}) \leq 0, \\
    1/2 + \theta_{p_l}^{(\mathcal{L})} \quad  \textrm{if} \quad \displaystyle\sum_{j =1}^{l -1} b(p_j)\exp(- 2 \pi i \theta_{p_j}^{(0)}) > 0.
  \end{cases}
\]
By the construction of $\underline{\theta}^{(0)} \in \mathbb{R}^{\mathbb{N}}$, 
we have the estimate \eqref{eqn:ALT}
for any $\xi > 0$ when the estimate $ | b ( p ) | = |a (p)| \leq C_{\LF, \eta} p^\eta $ holds for any prime number $p$ with some positive number $\eta$.
Taking $\theta_p = 0$ for $p \in \mathcal{P} \setminus \mathcal{P}_0$, 
we have the conclusion.
\end{proof}
We put
\[
\gamma_k
= \sum_{p}^\infty \sum_{l =1}^\infty \frac{(- \log p^l)^k b(p^l) \exp( - 2 \pi i l \theta_{p}^{(0)}) }{p^{l \sigma_0}}
\]
for any $k = 0, \ldots, N-1$ and $\underline{\gamma} = ( \gamma_k )_{k =0}^{N -1}$.
Since it holds that
\[
\sum_{p} \sum_{l =2}^\infty \frac{ | b(p^l) | ( \log p^l )^k }{p^{l \sigma_0}} < \infty
\]
by an argument similar to (2.13) in \cite{NS2010},
we find that the series $\gamma_k$ is convergent by partial summation.
\end{step}

\begin{step}
We will use the positive density method introduced by Laurin\u{c}ikas and Matsumoto \cite{LM2000}.
The following idea is due to \cite{NS2010}.
Put $\mu = \sqrt{\kappa/8}$ and $\rho =  \kappa/4$.
Then we have
\[
\kappa - \rho>0 \quad \textrm{and} \quad
\kappa - 2 \mu^2 - \rho>0.
\]
We first take a positive number $\eta$ so that
\begin{equation}\label{eqn:eta1}
0 <\eta < 1/2 (1 - E_{\LF}),
\end{equation}
which is chosen more precisely later (see \eqref{eqn:eta}).
Then there exists 
 $C_{\LF, \eta} >0$ such that $| a( p ) | \leq C_{\LF, \eta} p^{\eta}$ for any prime number $p$ by the axiom (i) of the Selberg class $\mathcal{S}$. 
Let $U$ be a large positive parameter depending on $\LF, \sigma_0, N, \eta$, 
and let $U^{1 - E_{\LF}} (\log U)^{D +1} \leq H \leq U$.
We put 
\[
\mathcal{M}_{\mu, j}^{(U,  H)} 
= \{p \in \mathcal{P} ~;~ 2^jU < p \leq 2^j U  + H , | a (p) | > \mu\}
\]
for $j = 0,1, \ldots N-1$.
\begin{claim}\label{claim:PDM}
We have
\begin{equation}
\#(\mathcal{M}_{\mu, j}^{(U, H)}) 
\gg_{\LF, N, \eta} \frac{ H }{ U^{2 \eta} \log U },  \label{eqn:PDM}
\end{equation}
where $\#(A)$ denotes the cardinality of the set A.
\end{claim}
\begin{proof}[Proof of claim \ref{claim:PDM}]
Put $\pi_\mu(x) = \#\{p \in \mathcal{P} ~;~ p \leq x , | a(p) | > \mu \}$.
When $\alpha > \beta \geq 1$, 
it holds that
\begin{equation}\label{eqn:UB}
\sum_{\alpha < p \leq \beta} | a(p) |^2
\leq (C_{\LF, \eta}^2\beta^{2\eta} - \mu^2)\left( \pi_\mu(\beta) - \pi_\mu(\alpha) \right)
+ \mu^2(\pi(\beta) - \pi(\alpha) )
\end{equation}
by $(2.26)$ in \cite{NS2010}. 
By the condition (C3), 
it holds that
\begin{equation}\label{eqn:PNTS1}
\pi(2^j U +  H) - \pi(2^j U)
\leq 2 \frac{H}{\log U}
\end{equation}
and
\begin{equation}\label{eqn:PNTS2}
\sum_{2^j U < p \leq 2^j U + H} | a(p) |^2
\geq (\kappa - \rho)\frac{H}{\log U}.
\end{equation}
Substituting $\alpha = 2^j U$ , $\beta = 2^j U +  H$ for $j=0,1,\ldots, N-1$ into \eqref{eqn:UB} and using the estimate \eqref{eqn:PNTS1} and \eqref{eqn:PNTS2},
we have
\begin{align*}
\#(\mathcal{M}_{\mu, j}^{(U, H)}) 
=& \pi_\mu(2^j U + H ) - \pi_\mu(2^j U) \\
\geq& (\kappa - 2 \mu^2 - \rho)\frac{H}{ (4^{N \eta} C_{\LF ,\eta}^2 U^{2\eta} - \mu^2) \log U}  \\ 
\gg&_{\LF, N, \eta} \frac{ H }{ U^{2 \eta} \log U }. 
\end{align*}
This completes the proof.
\end{proof}
\end{step}

\begin{step}
We will use Lemma \ref{lem:EG} in this step.
Let $X$ be a sufficiently large positive number depending on $\LF, \sigma_0, N, \eta$.
We may assume that $\#(\mathcal{M}_{\mu, 0}^{(X, X)})\geq N$ by \eqref{eqn:eta1} and \eqref{eqn:PDM}.
Fix the distinct primes $p_{k_0}, p_{k_2}, \ldots , p_{k_{N - 1}} \in \mathcal{M}_{\mu, 0}^{(X, X)}$.
Let $Y \geq 2X + 1$ be a positive parameter which is determined later (see \eqref{eqn:Y1}), 
and let $Y^{1 - E_{\LF}} (\log Y)^{D +1} \leq H \leq Y$ which is determined later (see \eqref{eqn:EU}).
Let
\[
\mathcal{M}_j^{(Y, H)}
= \left\{ p \in \mathcal{P} ~;~ 2^j Y < p \leq 2^j Y + H\right\}
\]
and put 
\[
\mathcal{M}_\mu^{(Y, H)} = \bigsqcup_{j =0}^{N -1} \mathcal{M}_{\mu, j}^{(Y, H)}
\quad \textrm{and} \quad
\mathcal{M}^{(Y, H)} = \bigsqcup_{j =0}^{N -1} \mathcal{M}_{ j}^{(Y, H)}.
\]
\begin{claim}\label{claim:PSL}
We consider the following two conditions;
\begin{enumerate}[$(i)$]
\item[{\rm (Y1)}] For any $p_X \in \mathcal{M}_{\mu, 0}^{(X, X)}$ and $p_Y \in \mathcal{M}^{(Y, H)}$, it holds that
\[
\frac{|a(p_X)|}{p_X^{\sigma_0}}
\geq \frac{|a(p_Y)|}{p_Y^{\sigma_0}}, 
\]
\item[{\rm (Y2)}] For any $p_X \in \mathcal{M}_{\mu, 0}^{(X, X)}$ and $j = 0, 1, \ldots , N-1$, 
it holds that
\[
\frac{|a(p_X)|}{p_X^{\sigma_0}}
\leq \sum_{p \in \mathcal{M}_{j}^{(Y, H)}} \frac{|a(p)|}{p^{\sigma_0}}.
\]
\end{enumerate}
Then the choice
\begin{equation}\label{eqn:Y1}
Y = \left(\frac{C_{\LF, \eta}}{\mu}\right)^{\frac{1}{\sigma_0 - \eta}} (2 X)^{\frac{\sigma_0}{\sigma_0 - \eta}},
\end{equation}
yields the condition {\rm (Y1)}, and the estimate 
\begin{equation}\label{eqn:sur2}
\frac{1}{X^{\sigma_0 - \eta}} 
\ll_{\LF, \sigma_0, N, \eta} \frac{H}{Y^{\sigma_0 + 2 \eta} (\log Y)^2}.
\end{equation}
yields the condition {\rm (Y2)}. $($We take $H$ suitably which satisfies the bound \eqref{eqn:sur2} later $($see \eqref{eqn:eta}$)$.$)$
\end{claim}
\begin{proof}[Proof of Claim \ref{claim:PSL}]
We first consider the condition (Y1).
Since the estimates
\[
\frac{|a(p_X)|}{p_X^{\sigma_0}} 
\geq \frac{\mu}{p_X^{\sigma_0}}
\geq \frac{\mu}{(2X)^{\sigma_0}}
\quad \textrm{and} \quad
\frac{|a(p_Y)|}{p_Y^{\sigma_0}} 
\leq \frac{C_{\LF, \eta}}{p_Y^{\sigma_0 - \eta}}
\leq \frac{C_{\LF, \eta}}{Y^{\sigma_0 - \eta }}
\]
hold for $p_X \in \mathcal{M}_{\mu, 0}^{(X, X)}$ and $p_Y \in \mathcal{M}^{(Y, H)}$, 
the choice $Y$ in \eqref{eqn:Y1} yields the condition (Y1).

Next we consider the condition (Y2). 
Now it holds that
\begin{equation}\label{eqn:sur1}
\sum_{p \in \mathcal{M}_{ j}^{(Y, H)}} \frac{|a(p)|}{p^{\sigma_0}}
\geq \sum_{p \in \mathcal{M}_{\mu, j}^{(Y, H)}} \frac{|a(p)|}{p^{\sigma_0}}
\geq \frac{\mu}{(2^N Y)^{\sigma_0}} \#\left( \mathcal{M}_{\mu, j}^{(Y, H)} \right)
\gg_{\LF, \sigma_0, N, \eta} \frac{H}{Y^{\sigma_0 + 2 \eta} \log Y}
\end{equation}
by \eqref{eqn:PDM}.
By using the above estimate and the estimate $|a(p_X)| p_X^{- \sigma_0} \leq C_{\LF, \eta}X^{-( \sigma_0 - \eta)}$ for $p_X \in \mathcal{M}_{\mu, 0}^{(X, X)}$, 
the condition (Y2) holds when the estimate $\eqref{eqn:sur2}$ holds.
\end{proof}
For any $j = 0, 1, \ldots , N - 1$, put 
\[
\mathcal{M}_j
= \{p_{k_j}\} \sqcup \mathcal{M}_j^{(Y, H)} 
\quad \textrm{and} \quad 
\mathcal{M} = \bigsqcup_{j =0}^{N -1} \mathcal{M}_j. 
\]
In what follows, we take 
\[
Y = \left(\frac{C_{\LF, \eta}}{\mu}\right)^{\frac{1}{\sigma_0 - \eta}} (2 X)^{\frac{\sigma_0}{\sigma_0 - \eta}}.
\]
Then we have
\begin{equation}\label{eqn:SL}
\left\{\sum_{p \in \mathcal{M}_j}\frac{b(p) \exp (- 2 \pi i \theta_p)}{p^{\sigma_0}} ~;~ (\theta_p)_{p \in \mathcal{M}_j} \in [0,1)^{ \mathcal{M}_j} \right\} 
=\left\{ z \in \mathbb{C} ~;~ | z | \leq \sum_{p \in \mathcal{M}_j} \frac{|b(p)|}{p^{\sigma_0}}  \right\}
\end{equation}
by Lemma \ref{lem:EG} and by (i) of Lemma \ref{lem:CS} when the estimate \eqref{eqn:sur2} holds.
\end{step}

\begin{step}
Let $\underline{\theta} = (\theta_p)_{p \in \mathcal{M}} \in \mathbb{R}^{\mathcal{M}}$ and write $\underline{\theta}_j = (\theta_p)_{p \in \mathcal{M}_j}$. 
We will prove the following claim. 
\begin{claim}\label{claim:FSE}
For $j, k = 0,1, \ldots N-1$ and for $\underline{\theta} = (\theta_p)_{p \in \mathcal{M}} \in \mathbb{R}^{\mathcal{M}}$, 
we have
\begin{equation}\label{eqn:FSE}
\frac{\partial^k}{\partial s^k}\phi_{\mathcal{M}_j}(\sigma_0 , \underline{\theta}_j)
= (- \log Y_j)^k \phi_{\mathcal{M}_j}(\sigma_0 , \underline{\theta}_j) + R_{j,k} , 
\end{equation}
\end{claim}
where $Y_j = 2^j Y$ and 
\begin{equation}\label{eqn:R1}
R_{j,k}
\ll_{\LF, \sigma_0, N, \eta} (\log Y)^{N -2} H^2 Y^{-(1 + \sigma_0 -\eta )}
+ \frac{(\log X)^{N -1}}{X^{\sigma_0 -\eta}}.
\end{equation}
\begin{proof}[Proof of Claim \ref{claim:FSE}]
By the equation \eqref{eqn:FSE}, we have
\begin{align*}
R_{j,k}
=& \sum_{p \in \mathcal{M}_j^{(Y, H)}}\left\{ (- \log p)^k - ( - \log Y_j )^k \right\}b(p)p^{- \sigma_0}\exp(-2\pi i \theta_p) \\
& \quad + \left\{ (- \log p_{k_j})^k - ( - \log Y_j )^k \right\}b(p_{k_j})p_{k_j}^{- \sigma_0}\exp( -2 \pi i \theta_{p_{k_j}} ).
\end{align*}
Using the mean value theorem for $(- \log p)^k - ( - \log Y_j )^k$ in the first term, 
we have
\begin{align*}
R_{j,k}
\ll_N \sum_{p \in \mathcal{M}_j^{(Y, H)}}(\log Y)^{k -1}\frac{H}{Y} | b( p ) | p^{-\sigma_0} 
\quad + |(\log p_{k_j})^k - ( \log Y_j )^k | |b(p_{k_j})| p_{k_j}^{-\sigma_0}.
\end{align*}
By using the estimate $|b(p)| = |a (p) | \ll_{\LF, \eta} p^\eta$, 
it holds that
\begin{align*}
\sum_{p \in \mathcal{M}_j^{(Y, H)}}(\log Y)^{k -1}\frac{H}{Y} | b( p ) | p^{-\sigma_0}
&\ll_{\LF, \eta} (\log Y)^{k -1} \frac{H}{Y} \sum_{p \in \mathcal{M}_j^{(Y, H)}} p^{\eta-\sigma_0} 
\\
&\leq (\log Y)^{k -1} \frac{H}{Y} Y^{\eta - \sigma_0}\sum_{p \in \mathcal{M}_j^{(Y, H)}}1
\\
&\ll (\log Y)^{N -2} H^2 Y^{-(1 + \sigma_0 -\eta )},
\end{align*}
and
\begin{align*}
&| (\log p_{k_j})^k - ( \log Y_j )^k | |b(p_{k_j})| p_{k_j}^{- \sigma_0} \\
&\quad \ll_{\LF, N, \eta} ( \log Y )^{N -1} p_{k_j}^{\eta - \sigma_0} 
\ll \frac{(\log Y)^{N -1}}{X^{\sigma_0 -\eta}}
\ll_{\LF, \sigma_0, N, \eta} \frac{(\log X)^{N -1}}{ X^{\sigma_0 - \eta} }.
\end{align*}
This completes the proof.
\end{proof}
\end{step}

\begin{step}
We now consider the following system of linear equations in the unknown $z_j$:
\begin{equation}\label{eqn:VDM}
\sum_{j = 0}^{ N-1 } (- \log Y_j)^k z_j = c_k - \gamma_k \quad \text{for $k = 0, 1, \ldots , N-1$}.
\end{equation}
Since the coefficient matrix of this system is the Vandermonde matrix, 
Lemma \ref{lem:Van} implies that it has a unique solution $\underline{z}=\underline{z}(Y, \underline{c}, \underline{\gamma}) = (z_0, z_1, \ldots , z_{N-1})$ which satisfies the bound
\begin{equation}\label{eqn:sur3}
\| \underline{z} \|
\ll_N (\log Y)^{N -1}\| \underline{c} - \underline{\gamma} \|.
\end{equation}
\begin{claim}\label{claim:SL}
A sufficient condition that the system of equations 
\begin{equation}\label{eqn:SE}
\phi_{\mathcal{M}_j}(\sigma_0 , \underline{\theta}_j) = z_j, \quad
\text{for}\quad 
j=0,1,\ldots ,N-1
\end{equation}
has a solution $\underline{\theta}\in\mathbb{R}^{\mathcal{M}}$
is that the estimate \eqref{eqn:sur2} and the estimate
\begin{equation}\label{eqn:sur4}
\frac{H}{ Y^{\sigma_0 + 2 \eta}( \log Y) ^{ N + 1 }}
\gg_{\LF, \sigma_0, N, \eta} \| \underline{c} - \underline{\gamma} \| +1.
\end{equation}
hold.
\end{claim}
\begin{proof}[Proof of Claim \ref{claim:SL}]
It is enough to take $H$ to establish 
\begin{equation}\label{eqn:Sur}
\| \underline{z} \|
\leq \sum_{p \in \mathcal{M}_j}\frac{| b(p) |}{p^{\sigma_0}}
\end{equation}
by \eqref{eqn:SL}.
The bound \eqref{eqn:sur1} and \eqref{eqn:sur3} give the proof.
\end{proof}
\end{step}

\begin{step}
Note that, by $\sigma_0 \in (\max\{ \sigma_{\LF} , 1 - 2 E_{\LF}  \},1)$, it holds that $\left(\sigma_0, \frac{1 + \sigma_0}{2} \right) \cap \left( 1 - E_{\LF}, 1 \right) \neq \emptyset$.
We will show the following claim.
\begin{claim}\label{claim:THE}
Let $H = Y^A$ and choose $A = A(\sigma_0, E_{\LF})$ and $\eta = \eta(\sigma_0, E_{\LF})$ such that
\begin{equation}\label{eqn:EU}
A = A(\sigma_0, E_{\LF}) = \frac{1}{2}\left( \max\left\{ \sigma_0, 1 - E_{\LF}\right\} + \frac{1 + \sigma_0}{2}  \right)
\in \left(\sigma_0, \frac{1 + \sigma_0}{2} \right) \cap \left( 1 - E_{\LF}, 1 \right)
\end{equation}
and
\begin{equation}\label{eqn:eta}
\eta = \eta(\sigma_0, E_{\LF}) 
= \frac{1}{2} \min \left\{ \frac{1 - E_{\LF}}{2}, \frac{A ( \sigma_0, E_{\LF} ) - \sigma_0}{2},  1 + \sigma_0 - 2 A(\sigma_0, E_{\LF}) \right\} >0. 
\end{equation}
Put 
\[
d_1^{(1)} (\sigma_0, E_{\LF})
= \frac{\sigma_0}{\sigma_0 - \eta} \left( A - \sigma_0 - 2 \eta \right) > 0
\]
and
\begin{equation*}
B ( \sigma_0, E_{\LF} ) 
= \min\left\{ \frac{\sigma_0}{ \sigma_0 - \eta } \left( 1 + \sigma_0 - 2 A - \eta \right), \sigma_0 - \eta \right\} > 0.
\end{equation*}
If the estimate
\begin{equation}\label{eqn:X}
X \geq C^{(1)}(\LF, \sigma_0, N) 
\left( \| \underline{c} - \underline{\gamma} \| _N+1 \right)^{ 2 (d_1^{(1)} (\sigma_0, E_{\LF}) )^{-1}}
\end{equation}
holds with sufficiently large $C^{(1)}(\LF, \sigma_0, N)$ depending on $\LF, \sigma_0, N$, 
then there exists $\underline{\theta}^{(1)} = \underline{\theta}^{(1)} (\LF, \sigma_0, N, X)= (\theta_p^{(1)})_{p \in \mathcal{M}} \in \mathbb{R}^{\mathcal{M}}$ such that
\begin{equation}\label{eqn:AP1}
\left| \sum_{j = 0}^{ N-1 }\frac{\partial^k}{\partial s^k} \phi_{\mathcal{M}_j}(\sigma_0 , \underline{\theta}^{(1)}_j) - ( c_k - \gamma_k ) \right|
\ll_{\LF, \sigma_0, N}
X^{- B (\sigma_0, E_{\LF})} (\log X)^{N -1}
\end{equation}
holds for any $k = 0, 1, \ldots , N-1$.
\end{claim}

\begin{proof}[Proof of Claim \ref{claim:THE}]
Let X satisfy the bound \eqref{eqn:X}.
Then we have \eqref{eqn:sur2}.
By substituting \eqref{eqn:Y1}, 
the left hand side of \eqref{eqn:sur4} equals
\begin{equation}\label{eqn:sur5}
\frac{H}{Y^{\sigma_0 + 2 \eta} (\log Y)^{N +1}}
= \frac{ Y^{ A  -\sigma_0 - 2\eta  } }{ (\log Y)^{N +1} }
\asymp_{\LF, \sigma_0, N} 
\frac{X^{d_1^{(1)} (\sigma_0, E_{\LF})}}{ ( \log X )^{N +1}},
\end{equation}
and the estimate
\begin{equation}\label{eqn:R2}
R_{j,k}
\ll_{\LF, \sigma_0, N} X^{- B (\sigma_0, E_{\LF})} (\log X)^{N -1}
\end{equation}
holds by \eqref{eqn:R1}.
By the estimate \eqref{eqn:sur5}, 
we have \eqref{eqn:sur4}.
Hence Claim \ref{claim:SL} with the estimates \eqref{eqn:sur2} and \eqref{eqn:sur4} implies that there exists $\underline{\theta}^{(1)} = \underline{\theta}^{(1)} (\LF, \sigma_0, N, X)= (\theta_p^{(1)})_{p \in \mathcal{M}} \in \mathbb{R}^{\mathcal{M}}$ such that the system of equations \eqref{eqn:SE} holds.
Therefore, by \eqref{eqn:FSE}, \eqref{eqn:VDM}, \eqref{eqn:SE} and \eqref{eqn:R2}, 
we have the conclusion.
\end{proof}
\end{step}

\begin{step}
To finish the proof, we give some estimates.
Put 
\[
\delta_0 = \frac{1}{2} ( \sigma_0 - 1/2 ), \quad
l_0 = \frac{2}{\sigma_0 - 1/2}.
\]
Then we have
\[
\left| \frac{\partial^k}{\partial s^k} \log \LF_{p}(\sigma_0 , \theta_p) - \frac{\partial^k}{\partial s^k} \phi_{p}(\sigma_0 , \theta_p) \right|\\
\leq \sum_{l = 2}^\infty \frac{ l^{N-1} | b (p^l) | ( \log p )^{N-1} }{p^{l \sigma_0}}.
\]
for $k = 0, 1, \ldots, N -1$ and $\theta_p \in \mathbb{R}$.
By a calculation similar to (2.12) in \cite{NS2010}, 
we have
\[
\sum_{l = 2}^\infty \frac{ l^{N-1} | b (p^l) | ( \log p )^{N-1} }{p^{l \sigma_0}}
\ll_{\LF, \sigma_0, N} \frac{(\log p)^{N-1}}{p^{2( \sigma_0 - \delta_0 )}} + \frac{( \log p )^{N-1}}{p^{ l_0 ( \sigma_0 - 1/2 ) }}
\ll \frac{( \log p )^{N-1}}{p^{\sigma_0 + 1/2}} .
\]
Hence it holds that
\begin{equation}\label{eqn:AP2}
\left| \frac{\partial^k}{\partial s^k} \log \LF_{p}(\sigma_0 , \theta_p) - \frac{\partial^k}{\partial s^k} \phi_{p}(\sigma_0 , \theta_p) \right|
\ll_{\LF, \sigma_0, N} \frac{( \log p )^{N-1}}{p^{\sigma_0 + 1/2}}
\end{equation}
for $k = 0, 1, \ldots, N -1$ and $\theta_p \in \mathbb{R}$.

From now, let $Q > 2^{N } Y$ and let $X$ satisfy \eqref{eqn:X}.
Put
\[
\theta_p^{(\star)}
= \begin{cases}
    \theta_p^{(0)} & \text{if $p \in \mathcal{P} \setminus \mathcal{M}$},  \\
    \theta_p^{(1)}  & \text{if $p \in\mathcal{M} $}.
  \end{cases}
\]
Then we have the following estimates.
\begin{claim}\label{claim:EEA}
We have
\begin{equation}\label{eqn:EEA1}
\frac{\partial^k}{\partial s^k} \log \LF_{ \mathcal{P}(Q) \setminus \mathcal{M} }(\sigma_0 , \underline{\theta}^{(0)})
=\gamma_k + O_{\LF, \sigma_0, N} \left( X^{1/2 - \sigma_0}  ( \log X )^{N -1} \right),
\end{equation}
\begin{equation}\label{eqn:EEA2}
 \frac{\partial^k}{\partial s^k} \log \LF_{\mathcal{M}}(\sigma_0 , \underline{\theta}^{(1)}) - \frac{\partial^k}{\partial s^k} \phi_{ \mathcal{M}}(\sigma_0 , \underline{\theta}^{(1)}) 
\ll_{\LF, \sigma_0, N} X^{1/2 - \sigma_0} (\log X)^{N-1}
\end{equation}
hold for any $k = 0, 1, \ldots , N-1$.
\end{claim}
\begin{proof}[Proof of Claim \ref{claim:EEA}]
The estimate \eqref{eqn:EEA2} follows from the estimate \eqref{eqn:AP2}.
Next, we will show the estimate \eqref{eqn:EEA1}.
We can write
\begin{align*}
\frac{\partial^k}{\partial s^k} \log \LF_{ \mathcal{P}(Q) \setminus \mathcal{M} }(\sigma_0 , \underline{\theta}^{(0)})
= \gamma_k - \left( \sum_{p > Q} + \sum_{p \in \mathcal{M}} \right) \frac{\partial^k}{\partial s^k} \log \LF_{ p }(\sigma_0 , \underline{\theta}^{(0)}).
\end{align*}
By the estimate \eqref{eqn:AP2}, we have
\begin{align*}
&\left( \sum_{p > Q} + \sum_{p \in \mathcal{M}} \right) \left( \frac{\partial^k}{\partial s^k} \log \LF_{p}(\sigma_0 , \underline{\theta}^{(0)}) - \frac{\partial^k}{\partial s^k} \phi_{p}(\sigma_0 , \underline{\theta}^{(0)}) \right) \\
\ll&_{\LF, \sigma_0, N} \sum_{p >X} \frac{(\log p)^{N -1}}{p^{\sigma_0 + 1/2}}
\ll _N X^{1/2} ( \log X )^{N -1}.
\end{align*}
On the other hand, we have
\begin{align*}
&(-1)^k\sum_{p > Q} \frac{\partial^k}{\partial s^k} \phi_{p}(\sigma_0 , \underline{\theta}^{(0)})
= \sum_{p > Q} \frac{b (p) (\log p)^k \exp( - 2 \pi i \theta_p^{(0)} ) }{p^{\sigma_0}} \\
& = \left[ \left( \sum_{p \leq \xi} b(p) \exp( - 2 \pi i \theta_p^{(0)} )  \right) \frac{( \log \xi )^k}{\xi^{\sigma_0}} \right]_{\xi =  Q}^{\xi = \infty} \\
& \quad - \int_{Q}^{\infty} \left( \sum_{p \leq \xi} b(p) \exp( - 2 \pi i \theta_p^{(0)} )  \right) d \left( \frac{( \log \xi )^k}{\xi^{\sigma_0}} \right) \\
&\ll_{\LF, \sigma_0, N} Q^{1/2 - \sigma_0} ( \log Q )^{N -1}
\leq X^{1/2 - \sigma_0} ( \log X )^{N -1}
\end{align*}
by partial summation and the estimate \eqref{eqn:ALT}.
By a calculation similar to the above, 
we have
\begin{align*}
&(-1)^k \sum_{p \in \mathcal{M}_j} \frac{\partial^k}{\partial s^k} \phi_{ p}(\sigma_0 , \underline{\theta}^{(0)}) \\
=& \frac{b (p_{k_j}) (\log p_{k_j})^k \exp( - 2 \pi i \theta_{p_{k_j}}^{(0)} ) }{p^{\sigma_0}} 
+ \sum_{Y_j < p \leq Y_j + H} \frac{b (p) (\log p)^k \exp( - 2 \pi i \theta_p^{(0)} ) }{p^{\sigma_0}} \\
\ll&_{\LF, \sigma_0, N} X^{1/2 - \sigma_0} ( \log X )^{N -1}
\end{align*}
for any $j = 0, 1, \ldots, N -1$. 
This completes the proof.
\end{proof}
\end{step}

\begin{step}
We will finish the proof in this step. 
By Claim \ref{claim:EEA},
we have
\begin{align*}
&\frac{\partial^k}{\partial s^k} \log \LF_{ \mathcal{P}(Q) }(\sigma_0 , \underline{\theta}^{(\star)}) \\
&=\frac{\partial^k}{\partial s^k}\phi_{ \mathcal{M}}(\sigma_0 , \underline{\theta}^{(1)}_j) 
+ \left( \frac{\partial^k}{\partial s^k} \log \LF_{\mathcal{M}}(\sigma_0 , \underline{\theta}^{(1)}) - \frac{\partial^k}{\partial s^k} \phi_{ \mathcal{M}}(\sigma_0 , \underline{\theta}^{(1)}) \right)\\
&\quad \quad + \frac{\partial^k}{\partial s^k} \log \LF_{ \mathcal{P}(Q) \setminus \mathcal{M}}(\sigma_0 , \underline{\theta}^{(0)}) \\
&=c_k - \gamma_k + O_{\mathcal{L}, \sigma_0, N} \left( X^{- B (\sigma_0, E_{\LF})} (\log X)^{N -1} \right) \\
&\quad \quad  + \gamma_k + O_{\mathcal{L}, \sigma_0, N} \left(X^{1/2 - \sigma_0}  ( \log X )^{N -1}\right) \\
&= c_k + O_{\LF, \sigma_0, N} \left( X^{- \min\left\{ B(\sigma_0, E_{\LF}), \sigma_0 - 1/2 \right\}}  ( \log X )^{N -1} \right)
\end{align*}
for any $k = 0, 1, \ldots , N-1$.
Hence we have
\[
\left| \frac{\partial^k}{\partial s^k} \log \LF_{ \mathcal{P}(Q) }(\sigma_0 , \underline{\theta}^{(\star)}) - c_k \right|
\ll_{\LF, \sigma_0, N} X^{- \min\left\{ B(\sigma_0, E_{\LF}), \sigma_0 - 1/2 \right\}}  ( \log X )^{N -1}
\]
for any $k = 0, 1, \ldots , N-1$.
Putting 
\[
d_1(\sigma_0, E_{\LF})
= \frac{2 \sigma_0}{\sigma_0 - \eta} \max\left\{ (d_1^{(1)} (\sigma_0, E_{\LF}) )^{-1}, \left(\min\left\{  B(\sigma_0, E_{\LF}), \sigma_0 - 1/2 \right\} \right)^{-1} \right\},
\]
letting $C (\LF,  \sigma_0, N ) $ be sufficiently large depending on $\LF,  \sigma_0, N$,
and using 
\[
\| \underline{c} - \underline{\gamma} \| + 1/ \epsilon  \ll_{\LF} \| \underline{c} \| + 1/\epsilon,
\]
we have the conclusion.
\end{step}
\vspace{-4mm}
\end{proof}

\begin{proof}[Proof of Theorem \ref{thm:MTH}]
We divide the proof into several steps.
\begin{stepA} 
First we will give settings and mention the strategy of the proof. 
Let $Q$ satisfy 
\[
Q> C_1^{(1)} ( \LF, \sigma_0, N ) \left( \| \underline{c} \| + 1 / \varepsilon  \right)^{d_1(\sigma_0, E_{\LF})},
\] 
where $C_1^{(1)} ( \LF, \sigma_0, N )$ is a sufficiently large constant depending on $\LF,  \sigma_0, N$ and satisfying $C_1^{(1)} ( \LF, \sigma_0, N ) \geq 2^{8/(\sigma_0 - 1/2)}$.
Let $\underline{\theta}^{(\star)} = (\theta_p^{(\star)})_{p \in \mathcal{P}(Q)} \in \mathbb{R}^{\mathcal{P}(Q)}$ be as in Lemma \ref{lem:AP} and put $\delta = Q^{-1}$.
We put
\[
\mathcal{I}
= \int_{D_T} \sum_{k = 0}^{N -1} \Phi_Q \left( \underline{\gamma}(t) - \underline{\theta}^{(\star)} \right)  \left| \left( \log \LF ( \sigma_0 + i t )  \right)^{(k)} - \left( \log \LF_{\mathcal{P}(Q)} ( \sigma_0 + i t )  \right)^{(k)} \right|^2 dt,
\]
where $\Phi_Q (\underline{\theta})$ is the mollifier defined in subsection \ref{sbs:MOF}, 
\[
\underline{\gamma}(t)
=\left( \frac{ \log p }{2 \pi} t  \right)_{p \in \mathcal{P}(Q)} \in \mathbb{R}^{\mathcal{P}(Q)},
\]
and $D_T$ is the subset of $[T,2T]$ which is defined as follows.
For each nontrivial zeros $\rho = \beta + i \gamma$ of $\LF(s)$, we define
\[
P^{(h)}_\rho
=\left\{ s = \sigma + it  ~;~ (1/2)(\sigma_{\LF} + \sigma_0) \leq \sigma \leq 15,\quad | t - \gamma | \leq h \right\}
\]
with the positive parameter $10 \leq h <T$, and we put
\[
D_T
=D_T (h)
=\left\{ t \in [T, 2T] ~;~ \sigma_0 + it \not \in \bigcup_{\rho~;~\beta>(1/2)(\sigma_{\LF} + \sigma_0) } P^{(h)}_\rho \right\}.
\]

Now, we mention the strategy of the proof.
We want to choose $T$ depending on $\LF, \sigma_0, N, \underline{c}, \epsilon$ and choose $Q$ and $h$ depending on $T$ so that
\begin{equation}\label{eqn:ST1}
\mathcal{I} 
\leq \left( \frac{\epsilon}{3} \right)^2 \int_{D_T} \Phi_Q \left( \underline{\gamma}(t) - \underline{\theta}^{(\star)} \right) dt,
\end{equation}
\begin{equation}\label{eqn:ST2}
\int_{D_T} \Phi_Q \left( \underline{\gamma}(t) - \underline{\theta}^{(\star)} \right) dt \geq \frac{T}{2}, 
\end{equation}
and 
\begin{equation}\label{eqn:ST3}
\left| \frac{\partial^k}{\partial s ^k} \log \LF_{\mathcal{P}(Q)} \left( \sigma_0, \underline{\theta} \right) - \frac{\partial^k}{\partial s ^k} \log \LF_{\mathcal{P}(Q)} \left( \sigma_0, \underline{\theta}^{(\star)} \right)\right| 
< \frac{\epsilon}{3}
\end{equation}
for $ | \theta_p - \theta_p^{(\star)} | < \delta $, $p \leq Q$ and $k = 0, 1, \ldots, N-1$.
Once we have such choices, there exists $t_0 \in [T, 2T]$ such that
\[
\left|\frac{\partial^k}{\partial s ^k} \log \LF ( \sigma_0 + i t_0 ) - \frac{\partial^k}{\partial s ^k} \log \LF_{\mathcal{P}(Q)} ( \sigma_0 + i t_0 ) \right|
\leq \frac{\epsilon}{3}
\]
for any $k = 0, 1, \ldots , N-1$ and $\int_{D_T} \Phi_Q \left( \underline{\gamma}(t_0) - \underline{\theta}^{(\star)} \right) dt > 0$.
By \eqref{eqn:ST3} and \eqref{eqn:SUPP} and by noting the equation $\log \LF_{\mathcal{P}(Q)} ( \sigma_0 + i t_0 ) = \log \LF_{\mathcal{P}(Q)} \left( \sigma_0, \underline{\gamma}(t_0) \right)$,
we have
\[
\left| \frac{\partial^k}{\partial s ^k} \log \LF_{\mathcal{P}(Q)} \left( \sigma_0 + i t_0 \right) - \frac{\partial^k}{\partial s ^k} \log \LF_{\mathcal{P}(Q)} \left( \sigma_0, \underline{\theta}^{(\star)} \right)\right| 
< \frac{\epsilon}{3}
\]
for any $k = 0, \ldots, N -1$ by substituting $\underline{\theta} = \underline{\gamma}(t_0)$.
These estimates and Lemma \ref{lem:AP} give the inequalities 
\[
\left| \frac{d^k}{d s^k} \log \LF (\sigma_0 + i t_0) - c_k \right| < \varepsilon \quad \textrm{for $k = 0, 1, \ldots, N -1$}.
\]
\end{stepA}

\begin{stepA}
We give a certain estimate toward the estimate \eqref{eqn:ST3}.
By using the estimates $| e^{i \alpha} - 1 | \leq | \alpha |$ for $\alpha \in \mathbb{R}$ and $b(p^l) \ll_{\LF} p^{l/2}$,
we have
\begin{align}
&\left| \frac{\partial^k}{\partial s ^k} \log \LF_{\mathcal{P}(Q)} \left( \sigma_0, \underline{\theta} \right) - \frac{\partial^k}{\partial s ^k} \log \LF_{\mathcal{P}(Q)} \left( \sigma_0, \underline{\theta}^{(\star)} \right)\right|
\ll \sum_{p\leq Q}\sum_{l =1}^\infty \frac{ l^{k +1} ( \log p )^k | b( p^l ) | }{p^{l \sigma_0}} \delta  \label{eqn:APX} \\ 
&\ll_{\LF, \sigma_0, N} Q^{-1} \sum_{p \leq Q} \frac{ ( \log p )^{N -1} }{p^{ \sigma_0 - 1/2 }} 
\ll_{\sigma_0, N} Q^{ 1/2 - \sigma_0 } ( \log Q )^{N -1 } \nonumber
\end{align}
for $ | \theta_p - \theta_p^{(\star)} | < \delta $, $p \leq Q$ and $k = 0, 1, \ldots, N-1$.
\end{stepA}

\begin{stepA}
We will estimate $\mathcal{I}$.
To estimate $\mathcal{I}$, we use the formula \eqref{eqn:SF} in Lemma \ref{thm:SF}.
First we will give the formula similar to \eqref{eqn:SF} for $\log \LF(s) $ .
Put
\[
F(s, z) = \int_{s+10}^s \frac{ x^{z-w} - x^{2(z-w)} }{(w-z)^2}dw.
\]
Integrating \eqref{eqn:SF}, we obtain
\begin{align}
&\log \LF(s) 
= \log \LF (s+10) 
+ \sum_{n \leq x^2}\frac{\Lambda_{\LF,x}(n)}{n^s \log n}
- \sum_{n \leq x^2}\frac{\Lambda_{\LF,x}(n)}{n^{s + 10} \log n} \nonumber \\
&\quad - \frac{m_\mathcal{L}}{\log x} F(s, 1) - \frac{m_\mathcal{L}}{\log x} F ( s, 0 )
+ \frac{1}{\log x}\sum_{\rho}F(s, \rho) 
+ \frac{1}{\log x} \sum_{j =1}^f \sum_{r =0}^\infty F(s, - q_j(r)) \nonumber \\
=& \sum_{n \leq x^2}\frac{\Lambda_{\LF,x}(n)}{n^s \log n} + \sum_{n > x} \left( \Lambda_{\LF}(n) - \Lambda_{\LF,x}(n) \right) \frac{1}{n^{s + 10} \log n} \label{eqn:logLF} \\
&\quad - \frac{m_\mathcal{L}}{\log x} F(s, 1) - \frac{m_\mathcal{L}}{\log x} F ( s, 0 )
+ \frac{1}{\log x}\sum_{\rho}F(s, \rho) 
+ \frac{1}{\log x} \sum_{j =1}^f \sum_{r =0}^\infty F(s, - q_j(r)) \nonumber
\end{align}
if $t$ is not equal to $0$ and the imaginary part of any zero of $\LF(s)$.
Let $Q < x\leq T$ and put
\[
\mathcal{I}_k
=\int_{D_T} \Phi_Q \left( \underline{\gamma}(t) - \underline{\theta}^{(\star)} \right)  \left| \left( \log \LF ( \sigma_0 + i t )  \right)^{(k)} - \left( \log \LF_{\mathcal{P}(Q)} ( \sigma_0 + i t )  \right)^{(k)} \right|^2 dt.
\]
We will estimate $\mathcal{I}_k$ for $k = 1, \ldots , N -1$.
For $k = 0$, we have the same upper bound by using the formula \eqref{eqn:logLF}.
For any $k = 1, 2, \ldots, N-1$, the estimate
\[
\mathcal{I}_k
\ll \mathcal{A}_k + \mathcal{B}_k + \mathcal{C}_k + \mathcal{D}_k + \mathcal{E}_k+ \mathcal{F}_k,
\]
holds, where
\begin{align*}
\mathcal{A}_k 
&= \int_{D_T} \Phi_Q \left( \underline{\gamma}(t) - \underline{\theta}^{(\star)} \right) \left| \left(\sum_{n \leq x^2}\frac{\Lambda_{\LF,x}(n)}{n^{s_0}}\right)^{(k-1)} - \left(\sum_{n \leq Q}\frac{\Lambda_{\LF}(n)}{n^{s_0}}\right)^{(k-1)} \right|^2 dt,\\
\mathcal{B}_k
&=\int_{D_T} \Phi_Q \left( \underline{\gamma}(t) - \underline{\theta}^{(\star)} \right) \left| \left( \sum_{p \leq Q} \sum_{l  > \frac{\log Q}{ \log p} } \frac{\Lambda_{\LF}(p^l)}{p^{l s_0}} \right)^{(k-1)} \right|^2 dt,\\
\mathcal{C}_k
&=\frac{1}{(\log x)^2}\int_{D_T} \Phi_Q \left( \underline{\gamma}(t) - \underline{\theta}^{(\star)} \right) \times \\
&\quad \quad \times \left| \left( \sum_{j =1}^f \sum_{r =0}^\infty \frac{ x^{- q_j(r) - s_0 } - x^{ - 2 (  q_j(r)  + s_0 ) } }{( q_j(r) + s_0 )^2} \right)^{(k-1)} \right|^2 dt,\\
\mathcal{D}_k
&=\frac{m_{\LF}^2}{(\log x)^2}\int_{D_T} \Phi_Q \left( \underline{\gamma}(t) - \underline{\theta}^{(\star)} \right) \left| \left( \frac{x^{2(1- s_0)} - x^{1- s_0} }{(1-s_0)^2 } \right)^{(k-1)} \right|^2 dt,\\
\mathcal{E}_k
&=\frac{m_{\LF}^2}{(\log x)^2}\int_{D_T} \Phi_Q \left( \underline{\gamma}(t) - \underline{\theta}^{(\star)} \right) \left| \left( \frac{x^{-2 s_0} - x^{- s_0 }}{ s_0 ^2 }  \right)^{(k-1)} \right|^2 dt,\\
\mathcal{F}_k
&=\frac{1}{(\log x)^2}\int_{D_T} \Phi_Q \left( \underline{\gamma}(t) - \underline{\theta}^{(\star)} \right) \left| \left( \sum_{\rho}\frac{x^{\rho - s_0} - x^{ 2 ( \rho - s_0 ) }}{(\rho - s_0)^2} \right)^{(k-1)} \right|^2 dt,
\end{align*}
and $s_0 = \sigma_0 + it.$

{\it Bound for $\mathcal{A}_k$.}
We can write
\begin{align*}
&\sum_{Q < n \leq x^2} \frac{\Lambda_{\LF,x}(n) (\log n)^{k -1}}{n^{s_0}}\\
=& \sum_{Q < p \leq x^2} \frac{\Lambda_{\LF,x}(p) (\log p)^{k -1}}{p^{s_0}}
+ \sum_{\substack{Q< n \leq x^2 ;\\ n = p^l, l \geq 2}} \frac{\Lambda_{\LF,x}(p^l) (\log p^l)^{k -1}}{p^{l s_0}},
\end{align*}
and
\begin{align}
&\sum_{\substack{Q< n \leq x^2 ;\\ n = p^l, l \geq 2}} \frac{\Lambda_{\LF,x}(p^l) (\log p^l)^{k -1}}{p^{l s_0}} \nonumber \\
\ll& \sum_{p \leq \sqrt{Q}} \sum_{l > \frac{\log Q}{\log p}}\frac{ | \Lambda_{\LF,x} (p^l)| (\log p^l)^{k -1}}{p^{l \sigma_0}}
+ \sum_{p>\sqrt{Q}} \sum_{l=2}^\infty \frac{ | \Lambda_{\LF,x}(p^l)| (\log p^l)^{k -1}}{p^{l \sigma_0}}.\label{eqn:EA2}
\end{align}
Note that
\begin{equation}\label{eqn:EP}
\sum_{l > X}\frac{l^k}{p^{l \sigma}}
\ll_{\sigma, N} \frac{X^{k}}{p^{X \sigma}}.
\end{equation}
holds for $X \geq 1$, $\sigma>0$ and for $k = 0, 1, \ldots N-1$.
We will estimate the first term of \eqref{eqn:EA2}.
By using the estimate $b(p^l) \ll_{\LF} p^{l/2}$ and \eqref{eqn:EP}, 
we have
\begin{equation}\label{eqn:EA21}
\sum_{l > \frac{\log Q}{\log p}}\frac{ | \Lambda_{\LF,x} (p^l)| (\log p^l)^{k -1}}{p^{l \sigma_0}}
= \sum_{l > \frac{\log Q}{\log p}} \frac{ | b(p^l) | ( \log p^l )^k }{p^{l \sigma_0}}
\ll_{\LF, \sigma_0, N} Q^{1/2 - \sigma_0} ( \log Q )^{N -1}
\end{equation}
for any $p \leq \sqrt{Q}$ and $k = 0, 1, \ldots, N-1$.
Put\[
\eta_0
= \frac{\sigma_0 - 1/2}{4}.
\]
Since Lemma \ref{lem:CS} (ii) yields the estimate
\[
b(p^l) 
\ll_{\LF, \sigma_0} ( 2^l - 1) \frac{p^{\eta_0 l}}{l}
\leq \frac{p^{\left( \eta_0 + \frac{\log 2}{\log p} \right) l}}{l},
\]
we have, by the estimate \eqref{eqn:EP},
\begin{align}
&\sum_{l > \frac{\log Q}{\log p}} \frac{ | b(p^l) | ( \log p^l )^k }{p^{l \sigma_0}}
\ll_{\LF, \sigma_0} (\log p)^k \sum_{l > \frac{\log Q}{\log p}} \frac{l^{k-1}}{ p^{ ( \sigma_0 - \eta_0 - \frac{\log 2}{\log p} ) l  } }\nonumber \\
\leq& (\log p)^k \sum_{l > \frac{\log Q}{\log p}} \frac{l^{k -1}}{ p^{(\sigma_0 - 1/2(\sigma_0 - 1/2)) l } } 
\ll_{\sigma_0, N} Q^{- \sigma_0 - 1/2(1/2 - \sigma_0)} (\log Q)^{N -1}\label{eqn:EA22}
\end{align}
for $p \geq 2^{4/(\sigma_0 - 1/2)}$.
By the estimate \eqref{eqn:EA21} and \eqref{eqn:EA22}, 
we have
\begin{align}
&\sum_{p \leq \sqrt{Q}} \sum_{l > \frac{\log Q}{\log p}}\frac{| \Lambda_{\LF,x}(p^l) | (\log p^l)^{k -1}}{p^{l \sigma_0}} \nonumber \\ 
=& \left( \sum_{2 \leq p < 2^{\frac{4}{\sigma_0 - 1/2 }}}  + \sum_{ 2^{\frac{4}{\sigma_0 - 1/2 }} \leq p \leq \sqrt{Q} } \right) \sum_{l > \frac{\log Q}{\log p}}\frac{| \Lambda_{\LF,x}(p^l) | (\log p^l)^{k -1}}{p^{l \sigma_0}} \nonumber \\ 
\ll&_{\LF, \sigma_0, N} Q^{1/2 - \sigma_0} ( \log Q )^{N -1}
+ \sum_{p \leq \sqrt{Q}} Q^{- \sigma_0 - 1/2(1/2 - \sigma_0)} (\log Q)^{N -1} \nonumber \\
\ll& Q^{1/2(1/2 - \sigma_0)} ( \log Q )^{N -1}. \nonumber
\end{align}
As for the second term of \eqref{eqn:EA2},
by an argument similar to \eqref{eqn:EA22},
we have
\begin{align*}
\sum_{p>\sqrt{Q}} \sum_{l=2}^\infty \frac{| \Lambda_{\LF,x}(p^l) | (\log p^l)^{k -1}}{p^{l \sigma_0}}
\ll_{\LF, \sigma_0, N}Q^{ (3/4) (1/2 - \sigma_0) } (\log Q)^{N -1}. 
\end{align*}
Hence we have
\begin{equation}\label{eqn:FBK}
\sum_{\substack{Q< n \leq x^2 ;\\ n = p^l, l \geq 2}} \frac{\Lambda_{\LF,x}(p^l) (\log p^l)^{k -1}}{p^{l s_0}} 
\ll_{\LF, \sigma_0, N} Q^{ 1/2(1/2 - \sigma_0) } (\log Q)^{N -1}. 
\end{equation}
Therefore, we have
\begin{align*}
\mathcal{A}_k
&\ll_{\LF, \sigma_0, N} \int_{D_T} \Phi_Q \left( \underline{\gamma}(t) - \underline{\theta}^{(\star)} \right) \left| \sum_{Q < p \leq x^2} \frac{\Lambda_{\LF,x}(p) (\log p)^{k -1}}{p^{s_0}} \right|^2 dt\\
&\quad  + Q^{1/2 - \sigma_0} (\log Q)^{2N -2} \int_{D_T} \Phi_Q \left( \underline{\gamma}(t) - \underline{\theta}^{(\star)} \right) dt
=: \mathcal{A}_k^{(1)} + \mathcal{A}_k^{(2)}.
\end{align*}
By the formula \eqref{eqn:MOL} with $\beta_{\underline{n}} := \beta_{\underline{n}} ( \underline{\theta}^{(\star)} )$, 
we have
\begin{align*}
&\mathcal{A}_k^{(1)} \\
\ll& \sum_{ \substack{ \underline{n} = (n_q)_{q \in \mathcal{P}(Q)} \in \mathbb{Z}^{\mathcal{P}(Q)} ; \\ \max_{q \in \mathcal{P}(Q)} |n_q| \leq M }} |\beta_{\underline{n}}| \left| \int_T^{2T} \exp\left( i t \sum_{q \in \mathcal{P}(Q)} n_q \log q \right) \left| \sum_{Q < p \leq x^2} \frac{\Lambda_{\LF,x}(p) (\log p)^{k -1}}{p^{s_0}} \right|^2 dt \right| \\
&+  \frac{1}{M} \exp( C_1 Q ) \int_T^{2T} \left| \sum_{Q < p \leq x^2} \frac{\Lambda_{\LF,x}(p) (\log p)^{k -1}}{p^{s_0}} \right|^2 dt \\
\ll& \int_T^{2T} \left| \sum_{Q < p \leq x^2} \frac{\Lambda_{\LF,x}(p) (\log p)^{k -1}}{p^{s_0}} \right|^2 dt \\
+& \sum_{ \substack{ \mathbf{0} \neq  \underline{n}  = (n_q)_{q \in \mathcal{P}(Q)} \in \mathbb{Z}^{\mathcal{P}(Q)} ; \\\max_{q \in \mathcal{P}(Q)} |n_q| \leq M }} |\beta_{\underline{n}}| \left| \int_T^{2T} \exp\left( i t \sum_{q \in \mathcal{P}(Q)} n_q \log q \right) \left| \sum_{Q < p \leq x^2} \frac{\Lambda_{\LF,x}(p) (\log p)^{k -1}}{p^{s_0}} \right|^2 dt \right| \\
&+  \frac{1}{M} \exp( C_1 Q ) \int_T^{2T} \left| \sum_{Q < p \leq x^2} \frac{\Lambda_{\LF,x}(p) (\log p)^{k -1}}{p^{s_0}} \right|^2 dt
=: \mathcal{A}_k^{(1,1)} + \mathcal{A}_k^{(1,2)} + \mathcal{A}_k^{(1,3)}.
\end{align*}
By using the estimates $| b(p) | = | a (p) | \ll_{\LF, \sigma_0} p^{\eta_0}$ and 
\[
\sum_{1 \leq n_1 \leq n_2 \leq T} \frac{1}{ n_1^{\alpha} n_2^{\alpha}  \log (n_2/n_1) }
\ll_{\alpha} T^{2 - 2\alpha} \log T
\quad \textrm{for} \quad 1/2 \leq \alpha < 1,
\]
we have
\begin{align*}
&\mathcal{A}_k^{(1,1)} \\
\leq& \sum_{Q < p \leq x^2} \frac{| b ( p ) |^2 ( \log p )^{2k} }{p^{2 \sigma_0}} T
+ O \left( \sum_{Q < p_1 <  p_2 \leq x^2} \frac{ | b(p_1) |  | b( p_2 ) | ( \log p_1 )^{k} ( \log p_2 )^{k} }{ p_1^{\sigma_0} p_2^{\sigma_0} \log (p_2 / p_1) }\right) \\
\ll&_{\LF, \sigma_0, N} Q^{(3/2) \left( 1/2 - \sigma_0 \right)} ( \log Q )^{ 2N - 2 } T 
+ x^{2} ( \log x )^{2N -1}.
\end{align*}
We also have
\begin{align*}
&\mathcal{A}_k^{(1,3)}
= \frac{1}{M} \exp( C_1 Q ) \mathcal{A}_k^{(1,1)}\\
\ll&_{\LF, \sigma_0, N} \frac{1}{M} \exp( C_1 Q ) \left( Q^{(3/2) \left( 1/2 - \sigma_0 \right)} ( \log Q )^{ 2N - 2 } T 
+ x^{2} ( \log x )^{2N -1} \right).
\end{align*}
We will estimate $\mathcal{A}_k^{(1,2)}$.
Fix $\mathbf{0} \neq \underline{n} = (n_q)_{q \in \mathcal{P}(Q)} \in \mathbb{Z}^{\mathcal{P}(Q)} $ with $\max_{q \in \mathcal{P}(Q)} |n_q| \leq M $. 
Then we have
\begin{align*}
&\left| \int_T^{2T} \exp\left( i t \sum_{q \in \mathcal{P}(Q)} n_q \log q \right) \left| \sum_{Q < p \leq x^2} \frac{\Lambda_{\LF,x}(p) (\log p)^{k -1}}{p^{s_0}} \right|^2 dt \right|\\
\leq& \sum_{Q < p \leq x^2} \frac{| b ( p ) |^2 ( \log p )^{2k} }{p^{2 \sigma_0}} \left| \int_T^{2T} \exp\left( i t \sum_{q \in \mathcal{P}(Q)} n_q \log q \right) dt \right| \\
+&O\Bigg{(} \sum_{Q < p_1 <  p_2 \leq x^2} \frac{ | b(p_1) | | b(p_2) | (\log p
_1)^k ( \log p_2 )^k }{p_1^{\sigma_0} p_2^{\sigma_0} } \times \\
& \quad \times \left| \int_T^{2T} \exp\left( i t \sum_{q \in \mathcal{P}(Q)} n_q \log q \right)  \left( \frac{p_2}{p_1} \right)^{it} dt \right| \Bigg{)} 
=: \widetilde{\mathcal{A}}_k^{(1,2,1)}(\underline{n}) + \widetilde{\mathcal{A}}_k^{(1,2,2)}(\underline{n})
\end{align*}
holds.
We will estimate
\[
\int_T^{2T} \exp\left( i t \sum_{q \in \mathcal{P}(Q)} n_q \log q \right) dt.
\]
Note that $\sum_{q \in \mathcal{P}(Q)} n_q \log q \neq 0.$
Put
\[
\mathcal{Q}^+(\underline{n})
=\{ q \in \mathcal{P}(Q) ~;~ n_q \geq 0 \}, \quad
\mathcal{Q}^- (\underline{n})
=\{ q \in \mathcal{P}(Q) ~;~ n_q < 0 \},
\]
and
\[
Q^+(\underline{n})
= \prod_{q \in \mathcal{Q}^+(\underline{n})} q^{n_q}, \quad
Q^-(\underline{n})
= \prod_{q \in \mathcal{Q}^-(\underline{n})} q^{- n_q}.
\]
Then it holds that
\[
\sum_{q \in \mathcal{P}(Q)} n_q \log q
= \log \left( \frac{Q^+(\underline{n})}{ Q^-(\underline{n}) } \right). 
\]
By the asymptotic formula $\sum_{p \leq X} \log p \sim X$ as $X \rightarrow \infty$,
we have
\begin{equation}
\max\{ Q^+(\underline{n}), Q^-(\underline{n}) \}
\leq \exp\left( M \sum_{q \in \mathcal{P}(Q)} \log q \right)
\leq \exp\left( C_2 Q M \right)
\end{equation}
for some absolute positive constant $C_2$.
Hence, since it holds that
\[
|\log (n/m) | > \frac{1}{\max\{m,n\}}
\]
for any distinct positive integers $m$ and $n$,
we have
\begin{equation}\label{eqn:EOD}
\int_T^{2T} \exp\left( i t \sum_{q \in \mathcal{P}(Q)} n_q \log q \right) dt
\ll \exp\left( C_2 Q M \right) . 
\end{equation}
Therefore we have
\begin{equation}\label{eqn:EA121}
\widetilde{\mathcal{A}}_k^{(1,2,1)}(\underline{n})
\ll_{\sigma_0, N}  Q^{(3/2) \left( 1/2 - \sigma_0 \right)} ( \log Q )^{ 2N - 2 }  \exp\left( C_2 Q M \right). 
\end{equation}
By calculations similar to \eqref{eqn:EA121},
we have
\begin{align*}
&\widetilde{\mathcal{A}}_k^{(1,2,2)}(\underline{n})\\
\ll& \sum_{Q < p_1 <  p_2 \leq x^2} \frac{ | b(p_1) | | b(p_2) | (\log p
_1)^k ( \log p_2 )^k }{p_1^{\sigma_0} p_2^{\sigma_0} } \left| \int_T^{2T}\left( \frac{  p_2 Q^+(\underline{n}) }{ p_1 Q^-(\underline{n}) } \right)^{it} dt \right| \\
\ll& \sum_{Q < p_1 <  p_2 \leq x^2} \frac{ | b(p_1) | | b(p_2) | (\log p
_1)^k ( \log p_2 )^k }{p_1^{\sigma_0} p_2^{\sigma_0} } \max\{ p_2 Q^+(\underline{n}), p_1 Q^-(\underline{n}) \} \\
\ll&_{\LF, \sigma_0, N} x^{4} ( \log x )^{2N - 2} \exp\left( C_2 Q M \right).
\end{align*}
Hence we have, by the estimate \eqref{eqn:EB},
\begin{align*}
\mathcal{A}_k^{(1,2)} 
\leq& \sum_{ \substack{ \mathbf{0} \neq  \underline{n} = (n_q)_{q \in \mathcal{P}(Q)} \in \mathbb{Z}^{\mathcal{P}(Q)} ; \\ \max_{q \in \mathcal{P}(Q)} |n_q| \leq M }} |\beta_{\underline{n}}| \left(\tilde{\mathcal{A}}_k^{(1,2,1)}(\underline{n}) + \tilde{\mathcal{A}}_k^{(1,2,2)}(\underline{n}) \right) \\
\ll&_{\LF, \sigma_0, N} \left( \sum_{ \substack{ \mathbf{0} \neq  \underline{n} \in \mathbb{Z}^{\mathcal{P}(Q)} ; \\ \max_{q \in \mathcal{P}(Q)} |n_q| \leq M }} |\beta_{\underline{n}}| \right) x^{4} ( \log x )^{2N - 2} \exp\left( C_2 Q M \right)\\
\ll& x^{4} ( \log x )^{2N - 2} \exp(C_0 Q) \exp\left( C_2 Q M \right).
\end{align*}
Therefore we have
\begin{align*}
&\mathcal{A}_k^{(1)}\\
\ll&_{\LF, \sigma_0, N} \left( 1 + \frac{1}{M}\exp( C_1 Q )  \right) \left( Q^{(3/2)(1/2 - \sigma_0)} ( \log Q )^{ 2N - 2 } T + x^{2} ( \log x )^{2N -1} \right)\\
&+ x^{4} ( \log x )^{2N - 2} \exp(C_0 Q) \exp\left( C_2 Q M \right) \\
\ll& Q^{(3/2)(1/2 - \sigma_0)} ( \log Q )^{ 2N - 2 } T
+ x^{4} ( \log x )^{2N - 2} \exp \exp\left( C_3 Q \right),
\end{align*}
where we take $M = \exp( 2 C_1 Q )$, and $C_3 > 2 C_1$ is a positive absolute constant.
Therefore we have
\begin{align*}
\mathcal{A}_k 
\ll_{\LF, \sigma_0, N}& Q^{(3/2)(1/2 - \sigma_0)} ( \log Q )^{ 2N - 2 } T \\
&+ x^{4} ( \log x )^{2N - 2} \exp \exp\left( C_3 Q \right) \\
&\quad + Q^{1/2 - \sigma_0} (\log Q)^{2N -2} \int_{D_T} \Phi_Q \left( \underline{\gamma}(t) - \underline{\theta}^{(\star)} \right)  dt.
\end{align*}

{\it Bound for $\mathcal{B}_k$.}
By calculations similar to \eqref{eqn:FBK}, we have
\[
\sum_{p \leq Q} \sum_{l  > \frac{\log Q}{ \log p} } \frac{b(p^l) ( \log p^l )^k}{p^{l s_0}}
\ll_{\LF, N, \sigma_0} Q^{ 1/2 (1/2 - \sigma_0) } (\log Q)^{N -1}.
\]
Hence we have
\begin{align*}
\mathcal{B}_k
\ll_{\LF, N, \sigma_0}  Q^{ 1/2 - \sigma_0  } (\log Q)^{2N -2} \int_{D_T} \Phi_Q \left( \underline{\gamma}(t) - \underline{\theta}^{(\star)} \right)  dt.
\end{align*}

We will estimate $\mathcal{C}_k$, $\mathcal{D}_k$, $\mathcal{E}_k$, $\mathcal{F}_k$ using the following bound.
We define
\begin{align*}
&\widetilde{D}_T
= \widetilde{D}_T (h) \\
=& \Bigg{\{} s ~;~ \sigma \geq \sigma_0 - 1/10( \sigma_0 - \sigma_{\LF} ), T - 1 \leq t \leq 2 T + 1, s \not \in\bigcup_{\rho~;~\beta>1/2(\sigma_{\LF} + \sigma_0) } P^{(h-1)}_\rho \Bigg{\}}.
\end{align*}
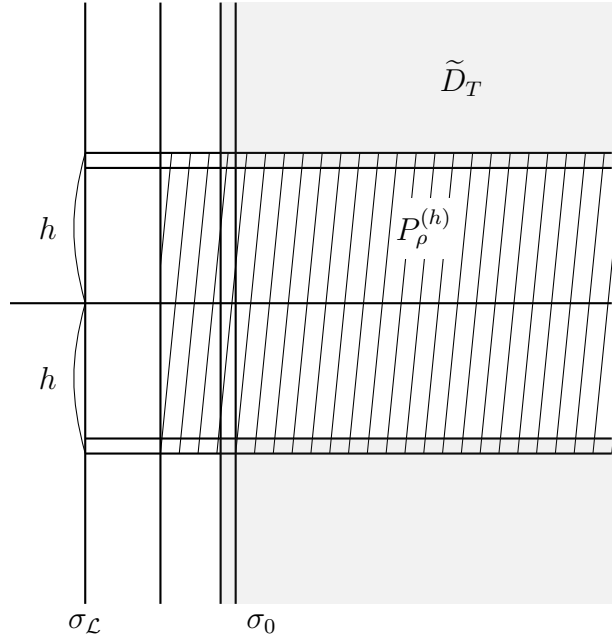
\begin{center}
\begin{tikzpicture}[scale=1]
\draw[thick] (0, -4) -- (0, 4) ; 
\draw (0, -4) node[below]{$ \sigma_{\mathcal{L}}$} ;
\draw[thick] (-1, 0) -- (7, 0) ; 
\draw[thick] (1, -4) -- (1, 4) ; 
\draw[thick] (2, -4) -- (2, 4) ; 
\draw (2, -4) node[below right]{$ \sigma_0$} ;
\draw[thick] (1.8, -4) -- (1.8, 4) ; 
\draw[thick] (0, 2) -- (7 ,2) ; 
\draw[thick] (0, 1.8) -- (7 ,1.8) ; 
\draw[thick] (0, -2) -- (7 ,-2) ; 
\draw[thick] (0, -1.8) -- (7 ,-1.8) ; 
\begin{scope}
\path[clip] (1, -2) -- (1, 2) -- (7, 2) -- (7, -2) -- cycle;
\foreach \t in {1,2,...,100}{
    \path[draw] (0.25*\t -3 , -2) -- (0.25*\t + 0.4 -3, 2);
    }
\end{scope}
\fill[gray, opacity=.1] (1.8, 1.8) -- (7, 1.8) -- (7, 4) -- (1.8, 4) ;
\fill[gray, opacity=.1] (1.8, -1.8) -- (7, -1.8) -- (7, -4) -- (1.8, -4) ;
\draw (0, 2) to [out= 255, in = 105] (0,0);
\draw (0, 0) to [out= 255, in = 105] (0,-2);
\draw (-0.5,1) node{$h$};
\draw (-0.5,-1) node{$h$};
\draw (5,3) node{$\widetilde{D}_T$};
\fill[white] (4.2, 0.6) -- (4.9, 0.6) -- (4.9, 1.4) -- (4.2, 1.4) ;
\draw (4.5,1) node{$P_{\rho}^{(h)}$};
\end{tikzpicture}
\end{center}
Then, by using the Cauchy integral formula, the estimate
\begin{equation}\label{eqn:CF}
| g^{(k -1)}(s_0) |
=\left| \frac{(k  -1)!}{2 \pi i} \int_{| z - s_0 | = 1/10( \sigma_0 - \sigma_{\LF} )} \frac{g (z)}{ ( z - s_0 )^k } d z \right| 
\ll_{\LF, \sigma_0, N} \sup_{z \in \widetilde{D}_T } | g (z) |
\end{equation}
holds for any holomorphic function $g(s)$ on $\widetilde{D}_T$ and $t \in D_T$ and for any $k = 1, \ldots, N -1$.

{\it Bound for $\mathcal{C}_k$, $\mathcal{D}_k$, $\mathcal{E}_k$}.
By using the estimate \eqref{eqn:CF}, we can easily check that
\begin{align*}
&\mathcal{C}_k
\ll_{\LF, \sigma_0, N} \frac{1}{T^4 x^{2 A(\sigma_0)} (\log x)^2} \int_{D_T} \Phi_Q \left( \underline{\gamma}(t) - \underline{\theta}^{(\star)} \right) dt,\\
&\mathcal{D}_k
\ll_{\LF, \sigma_0, N}
\frac{ x^{4 (1- A(\sigma_0))}}{T^4(\log x)^2} \int_{D_T} \Phi_Q \left( \underline{\gamma}(t) - \underline{\theta}^{(\star)} \right) dt,\\
&\mathcal{E}_k
\ll_{\LF, \sigma_0, N}
\frac{1}{T^4 x^{2 A(\sigma_0)} (\log x)^2} \int_{D_T} \Phi_Q\left( \underline{\gamma}(t) - \underline{\theta}^{(\star)} \right) dt,
\end{align*}
where $A(\sigma_0) = \sigma_0 - 1/10( \sigma_0 - \sigma_{\LF} )$.

{\it Bound for $\mathcal{F}_k$}.
We will estimate
\[
\sup_{s \in \widetilde{D}_T }  \left|  \sum_{\rho}\frac{x^{\rho - s} - x^{ 2 ( \rho - s ) }}{(\rho - s)^2} \right|^2.
\]
Fix $s = \sigma + it \in \widetilde{D}_T$.
We divide the sum into two sums;
\begin{align*}
&\sum_{\rho}\frac{x^{\rho - s} - x^{ 2 ( \rho - s ) }}{(\rho - s)^2} \\
=\ & \sum_{\substack{\rho;\\
0 < \beta \leq (1/2)( \sigma_{\LF} + \sigma_0 )}}\frac{x^{\rho - s} - x^{ 2 ( \rho - s ) }}{(\rho - s)^2}
+ \sum_{\substack{\rho;\\
(1/2)( \sigma_{\LF} + \sigma_0 ) < \beta \leq 1 }} \frac{x^{\rho - s} - x^{ 2 ( \rho - s ) }}{(\rho - s)^2} \\
=:& ~ \Sigma_L + \Sigma_R.
\end{align*}
By using the estimate $N_{\LF}(T + 1) - N_{\LF}(T) \ll_{\LF} \log T$ for $T \geq 2$ which is deduced by \eqref{eqn:GRV}, 
we have
\begin{align*}
& \Sigma_R
\ll x^{2(1 - \sigma )} \sum_{m \geq h -2} \sum_{m \leq | \gamma - t | < m +1} \frac{1}{| \gamma - t |^2} \\
&\leq x^{2(1 - \sigma )} \sum_{m \geq h -2} \frac{1}{m^2} \sum_{m \leq | \gamma - t | < m +1} 1 
\ll_{\LF}  x^{2(1 - \sigma )} \sum_{m \geq h -2} \frac{\log ( t + m )}{m^2}  \\
&\ll x^{2(1 - \sigma )} \sum_{m \geq h - 2} \frac{ \max\{ \log t , \log m \}}{m^2} 
\ll \frac{ x \log T}{h}. 
\end{align*}
On the other hand, 
we have
\begin{align*}
\Sigma_L
\ll_{\sigma_0} x^{- 2/5 ( \sigma_0 - \sigma_{\LF} )} \sum_{\rho} \frac{1}{1 + (t - \gamma)^2} 
\ll_{\LF} x^{- 2/5 ( \sigma_0 - \sigma_{\LF} )} \log T.
\end{align*}
Therefore we obtain
\begin{align*}
&\mathcal{F}_k \\
\ll&_{\LF, \sigma_0, N} \frac{1}{ (\log x)^2} \left( \frac{x^{2} (\log T)^2}{h^2}  +  x^{- 4/5 ( \sigma_0 - \sigma_{\LF} )} (\log T)^2 \right) \int_{D_T} \Phi_Q \left( \underline{\gamma}(t) - \underline{\theta}^{(\star)} \right) dt.
\end{align*}
\end{stepA}

\begin{stepA}\label{step:ST21}
Next we will give the lower bound for $\int_{D_T}   \Phi_Q \left( \underline{\gamma}(t) - \underline{\theta}^{(\star)} \right) dt $.
By the estimate \eqref{eqn:EB}, 
it holds that
\begin{align*}
&\int_{D_T}   \Phi_Q \left( \underline{\gamma}(t) - \underline{\theta}^{(\star)} \right) dt \\
=& \int_T^{2T}   \Phi_Q \left( \underline{\gamma}(t) - \underline{\theta}^{(\star)} \right) dt 
- \int_{[T, 2T] \setminus D_T}  \Phi_Q \left( \underline{\gamma}(t) - \underline{\theta}^{(\star)} \right) dt \\
\geq& \int_T^{2T} \Phi_Q \left( \underline{\gamma}(t) - \underline{\theta}^{(\star)} \right) dt 
- 2 h \sum_{\underline{n} \in \mathbb{Z}^{\mathcal{P}(Q)} } | \beta_{\underline{n}} | N_{\LF} \left( 1/2 ( \sigma_{\LF} + \sigma_0), 2 T \right) \\
=& \int_T^{2T}  \Phi_Q \left( \underline{\gamma}(t) - \underline{\theta}^{(\star)} \right) dt 
+ O \left(  h \exp ( C_0 Q ) N_{\LF} \left( 1/2 ( \sigma_{\LF} + \sigma_0), 2 T \right) \right).
\end{align*}
By the estimates \eqref{eqn:EB}, \eqref{eqn:MOL} and \eqref{eqn:EOD} and by substituting $M = \exp( 2 C_1 Q )$, 
we have
\begin{align*}
&\int_T^{2T}  \Phi_Q \left( \underline{\gamma}(t) - \underline{\theta}^{(\star)} \right) dt \\
\geq& T 
- \sum_{ \substack{ \mathbf{0} \neq\underline{n} \in \mathbb{Z}^{\mathcal{P}(Q)} ; \\\max_{p \leq Q} |n_p| \leq M }} | \beta_{\underline{n}} | 
\left| \int_{T}^{2T}  \exp \left( i t \sum_{p \leq Q } n_p \log p \right) dt \right| 
+ O\left( \frac{T}{M} \exp( C_1 Q ) \right)\\
=& T 
+ O\left( \exp \exp (C_3 Q) \right) 
+ O\left( \frac{T}{M} \exp( C_1 Q ) \right)\\
=& \left(  1 + O\left( \frac{ \exp \exp (C_3 Q) }{T} \right) + O \left( \frac{1}{ \exp(C_1 Q) } \right) \right) T .
\end{align*}
Taking $h = T^{\Delta_{\LF} ( 1/2 ( 1/2 + \sigma_0) ) / 2}$,
we have
\[
 h \exp ( C_0 Q ) N_{\LF} \left( 1/2 ( \sigma_{\LF} + \sigma_0), 2 T \right)
 \ll_{\LF, \sigma_0} \exp(C_0 Q) T^{ 1- \frac{\Delta_{\LF} ( 1/2 ( \sigma_{\LF} + \sigma_0) )}{2} }
\]
by the condition (C2).
Hence we have
\begin{align*}
\int_{D_T} \Phi_Q \left( \underline{\gamma}(t) - \underline{\theta}^{(\star)} \right) dt 
&\geq \Bigg{(}  1 + O_{\LF, \sigma_0} \left( \exp(C_0 Q) T^{- \frac{\Delta ( 1/2 ( \sigma_{\LF} + \sigma_0) )}{2} }  \right) + \\
& \quad \quad +  O\left( \frac{ \exp \exp (C_3 Q) }{T} \right) + O \left( \frac{1}{ \exp(C_1 Q) } \right) \Bigg{)} T.
\end{align*}
Taking $T = \exp\exp (C Q )$, $C = C_0 + C_3$, 
we have the inequality \eqref{eqn:ST2}.
\end{stepA}

\begin{stepA}
We will finish the proof in this step.
By the estimate \eqref{eqn:ST2} and by the estimates $\mathcal{A}_k$-$\mathcal{F}_k$, 
we obtain
\begin{align*}
\mathcal{I}
&\ll_{\LF, \sigma_0, N} \Big{(} Q^{1/2 - \sigma_0} ( \log Q )^{2N -2}+ \frac{x^4 ( \log x )^{2 N -2}}{T^{1/2}}  + \frac{x^2 ( \log T )^2}{T^{\Delta_{\LF} ( 1/2 ( 1/2 + \sigma_0) )}} +  \\
& \quad \quad + x^{- 4/5 ( \sigma_0 - \sigma_{\LF} )} (\log T)^2 \Big{)} \int_{D_T} \Phi_Q \left( \underline{\gamma}(t) - \underline{\theta}^{(\star)} \right) dt.
\end{align*}
Taking
\[
x= T^{\mu} \quad \textrm{with} \quad
\mu = \min\left\{ 1/200 ,\Delta ( 1/2 ( \sigma_{\LF} + \sigma_0) ) /10 \right\},
\]  
we have
\begin{equation}\label{eqn:ST11}
\mathcal{I}
\ll_{\LF, \sigma_0, N} Q^{1/2 ( 1/2 - \sigma_0 )} \int_{D_T} \Phi_Q \left( \underline{\gamma}(t) - \underline{\theta}^{(\star)} \right) dt .
\end{equation}
Put
\[
d(\sigma_0, E_{\LF})
= \max\left\{ {d_1(\sigma_0, E_{\LF}), \frac{8}{\sigma_0 - 1/2}} \right\}.
\]
By the estimates \eqref{eqn:APX} and \eqref{eqn:ST11} and step \ref{step:ST21}, 
we have \eqref{eqn:ST1}, \eqref{eqn:ST2} and \eqref{eqn:ST3} 
if 
\[
Q \geq C_1^{(1)} (\LF, \sigma_0, N) \left( \| \underline{c} \|_N + 1 / \varepsilon  \right)^{d(\sigma_0, E_{\LF})}
\]
holds.
Taking $C_1 (\LF, \sigma_0, N)=C C_1^{(1)}(\LF, \sigma_0, N)$, 
we have the conclusion.
\end{stepA}
\vspace{-4mm}
\end{proof}

\subsection{Proof of Corollary \ref{cor:CV}}
We will prove by using the same method as in \cite{V1988} and \cite{KV1992}.
Before stating the proof, 
we give some notations and lemmas.

Let $\mathbb{C}[[X]]$ be the formal power series ring with coefficients in $\mathbb{C}$ and indeterminate $X$.
We refer the reader to \cite{AIK2014} for the details of the general theory of the formal power series ring for example.
In what follows, we will use the following notations;
\begin{itemize}
\item Let $\mathbb{R}_+$ denote the set of nonnegative real numbers.
\item For any $N \in \mathbb{N}$, we write $[N]=[1, N] \cap \mathbb{N}$.
\item For empty index $\emptyset \in \mathbb{C}^{0}$, we adopt the convention that $\| \emptyset \|_0 = 0$.
\item Let $N\in\mathbb{N}$ and $[N]\subset A \subset \mathbb{N}$. 
For $\bm{z} = (z_j)_{j \in A} \in \mathbb{C}^A$, 
we write $\bm{z}_{[N]} = (z_j)_{j= 1}^N \in \mathbb{C}^N$.

\item Let $\alpha(X) = \sum_{n =0}^\infty a_n X^n \in \mathbb{C}[[X]]$ and $\beta(X) = \sum_{n =0}^\infty b_n X^n \in \mathbb{C}[[X]]$ with $b_n \in \mathbb{R}_{+}$ for any $n \in \mathbb{N}_0$.
The statement $\alpha \unlhd \beta$ means that $|a_n| \leq b_n$ holds for any $n \in \mathbb{N}$.
We also define $\alpha^{\abs}(X) \in \mathbb{C}[[X]]$ by $\alpha^{\abs} (X) = \sum_{n =0}^\infty |a_n| X^n$.
\item Let $\underline{Z} = (Z_j)_{j =1}^\infty \in \mathbb{C}^\mathbb{N}$ be the indeteminate and let $N$ be a positive integer. 
For a multi-index $\underline{i} = (i_1, \ldots, i_N) \in \mathbb{N}_0^N$ and for $f \left( \underline{Z} \right) \in \mathbb{C}[\underline{Z}] $,
define the symbol $\partial^{\underline{i}}$ as a differential operator given by
\[
\partial^{\underline{i}}\left(f (\underline{Z}) \right)
= \frac{\partial^{|\underline{i}|} f }{\partial Z_{1}^{i_1} \cdots \partial Z_{N}^{i_N}} (\underline{Z}) , \quad 
|\underline{i}| = i_1 + \cdots + i_N.
\]
In addition, 
we write $\underline{i} ! = i_1 ! \cdots i_N !$ and $(\underline{Z}_{[N]})^{\underline{i}}= Z_1^{i_1} \cdots Z_N^{i_N}$ .
\end{itemize}

Let $\underline{Z} = (Z_j)_{j =1}^\infty$ and $\underline{W} = (W_j)_{j =1}^\infty$ be the indeterminates and write $\underline{Z}_{[n]} = (Z_j)_{j =1}^n$ and $\underline{W}_{[n]} = (W_j)_{j =1}^n$ for $n \in \mathbb{N}$.
For $n \in \mathbb{N}$, 
define the polynomials $F_n(\underline{Z}_{[n]}) \in \mathbb{R}[\underline{Z}_{[n]}]$ and $G_n(\underline{W}_{[n]}) \in \mathbb{R}[\underline{Z}_{[n]}]$ by
\begin{equation}\label{eqn:rel}
\exp\left( \sum_{n = 1}^\infty Z_n X^n \right)
= 1 + \sum_{n =1}^\infty F_n(\underline{Z}_{[n]}) X^n
\end{equation}
and
\[
\log \left( 1 + \sum_{n =1}^\infty W_n X^n \right) = \sum_{n =1}^\infty G_n(\underline{W}_{[n]}) X^n.
\]
Then we have $\deg(F_n) =n$.
We also define the maps 
\[
\underline{F}:\mathbb{C}^\mathbb{N} \ni \underline{z} = (z_j)_{j =1}^\infty \mapsto \underline{F}(\underline{z}) = ( F_j (\underline{z}_{j}) )_{j =1}^\infty \in \mathbb{C}^\mathbb{N}, 
\]
\[
\underline{F}_{[N]}:\mathbb{C}^N \ni \bm{z} = (z_j)_{j =1}^N \mapsto \underline{F}_{[N]}(\bm{z}) = ( F_{j} (\bm{z}_{{[j]}}) )_{j =1}^N \in \mathbb{C}^\mathbb{N}, 
\]
and
\[
\underline{G}:\mathbb{C}^\mathbb{N} \ni \underline{w} = (w_j)_{j =1}^\infty \mapsto \underline{G} (\underline{w}) = ( G_j (\underline{w}_{[j]}) )_{j =1}^\infty \in \mathbb{C}^\mathbb{N}.
\]
We can easily check that $\underline{G}$ is the inverse mapping of $\underline{F}$.
For complex variables $\underline{z} = (z_j)_{j =1}^{\infty}$,
define 
\[
f(X;\underline{z}) = \sum_{n =1}^\infty z_n X^n, 
\]
\[
g(X, \underline{z}) = \log \left( 1 + \sum_{n =1}^\infty z_n X^n \right)
\]
and 
\[
h(X;\underline{z}) = - \log \left( 1 - \sum_{n =1}^\infty |z_n| X^n \right).
\]
Note that the relations
\begin{equation}\label{eqn:rel1}
f(X;\underline{z}) = g(X, \underline{F}(\underline{z}))
\quad \textrm{and} \quad g(X, \underline{z})
\unlhd h(X;\underline{z})
\end{equation}
hold.
\begin{lemma}\label{lem:BE}
Let $\underline{z}= (z_j)_{j=1}^\infty$ be complex variables.
\begin{itemize}
\item[{\rm (i)}] We have $f^{\abs}(X;\underline{z}) \unlhd h (X; \underline{F}(\underline{z}))$.
\item[{\rm (ii)}] Let $N \in \mathbb{N}$ and $\underline{i} = (i_1, \ldots, i_{N -1}) \in \mathbb{N}_0^{N -1}$ with $| \underline{i}  | \geq 1$. 
Then we have $f^{\abs}(X;\partial^{\underline{i}}\underline{F}(\underline{z})) \unlhd X^{S(\underline{i} )} \exp\left( f^{\abs} (X; \underline{z}) \right)$, 
where $\partial^{\underline{i}}\underline{F}(\underline{z}) = ( \partial^{\underline{i}} F_j (\underline{z}_{[j]}) )_{j =1}^\infty$ and $S(\underline{i} ) = i_1 + 2 i_2 + \cdots (N -1) i_{N -1}$.
\end{itemize}
\end{lemma}
\begin{proof}
Note that $\alpha (X) \unlhd \beta (X)$ implies $\alpha^{\abs}(X) \unlhd \beta(X)$ for $\alpha(X) , \beta(X) \in \mathbb{C}[[X]]$. 
This and the relation \eqref{eqn:rel1} deduce the first assertion. 
We fix $N \in \mathbb{N}$ and $\underline{i} = (i_1, \ldots, i_{N -1}) \in \mathbb{N}_0^{N -1}$ with $| \underline{i}  | \geq 1$.
By differentiating $\exp\left( f(X;\underline{Z}) \right)$, 
we have
\[
\frac{\partial }{\partial Z_j} \exp\left( f(X;\underline{Z}) \right)
= X^j \exp\left( f(X; \underline{Z}) \right).
\]
Hence, applying $\partial^{\underline{i}}$ on the both sides of the equation \eqref{eqn:rel} with $\underline{z}$ in place of $\underline{Z}$,
we obtain
\[
X^{ S(\underline{i}) } \exp \left( f(X;\underline{z}) \right) 
= f(X;\partial^{\underline{i}}\underline{F}(\underline{z})). 
\]
We can confirm that $\exp \left( f(X;\underline{z}) \right) \unlhd \exp \left( f^{\abs}(X;\underline{z}) \right)$,
and this completes the proof.
\end{proof}

\begin{lemma}\label{lem:BE1}
Fix $\epsilon>0$ and $N \in \mathbb{N}$.
Let $z_0, a_0 \in \mathbb{C}$ and $\bm{z} = (z_1, z_2, \ldots, z_{N -1})$, $\bm{\alpha} = (\alpha_1, \ldots, \alpha_{N -1})\in \mathbb{C}^{N-1}$.
\begin{itemize}
\item[{\rm (i)}] We have 
\[
\| \bm{\alpha} \|
\ll_N \left( 1 + \| \underline{F}_{[N -1]}(\bm{\alpha}) \| \right)^{N -1}
\]
\item[{\rm (ii)}] If $\| (z_0, \bm{z}) - (\alpha_0, \bm{\alpha}) \| < \delta < 1$, then we have
\[
\left| e^{z_0} - e^{\alpha_0} \right| < 2 | e^{\alpha_0} | \delta
\]
and 
\begin{equation*}
\left\| e^{z_0} \underline{F}_{[N-1]} ( \bm{z} ) - e^{\alpha_0} \underline{F}_{[N-1]} ( \bm{\alpha} ) \right\| 
 \ll_N | e^{\alpha_0} | \left( 1 + \left\| \underline{F}_{[N -1]} (\bm{\alpha}) \right\| \right)^{(N-1)^2} \delta.
\end{equation*}
\end{itemize}
\end{lemma}
\begin{proof}
Let 
\[
\underline{\bm{\alpha}^{(0)}} = (\alpha_1, \ldots, \alpha_{N -1}, 0, 0, \ldots) \in \mathbb{C}^\mathbb{N} 
\]
and
\[
\underline{w^{(0)}} = (F_1(\bm{\alpha}_{[1]}), F_2(\bm{\alpha}_{[2]}), \ldots, F_{N -1}(\bm{\alpha}), 0, 0, \ldots) \in \mathbb{C}^\mathbb{N}.
\]
Note that
\[
\underline{G}\left( \underline{w^{(0)}} \right)_{[N -1]} = \bm{\alpha},
\]
holds.
Then we have
\begin{align}
&\sum_{n =1}^{N -1} | \alpha_n | X^n
= f^{\abs}\left( X;  \underline{\alpha^{(0)}} \right)
\unlhd f^{\abs}\left( X; \underline{G}(  \underline{w^{(0)}}  ) \right)\label{eqn:1071} \\
\unlhd&\, h\left( X; \underline{F}\left( \underline{G}(  \underline{w^{(0)}}  )\right) \right)
= h\left(X; \underline{w^{(0)}} \right)
= - \log \left( 1 - \sum_{n =1}^{N -1} | F_n ( \bm{\alpha}_{[n]} ) | X^n \right) \nonumber
\end{align}
by (i) of Lemma \ref{lem:BE}.
Evaluating \eqref{eqn:1071} at $X = \left(3 \left(1 + \left\| F_{[N -1]} ( \bm{\alpha} ) \right\| \right) \right)^{-1}$, 
we have
\begin{align*}
\left( \frac{1}{3 \left(1 + \left\| F_{[N -1]} ( \bm{\alpha} ) \right\| \right)}  \right)^{N -1} \sum_{n =1}^{N -1} |\alpha_n|
&\leq \sum_{n =1}^{N -1} |\alpha_n| \left(\frac{1}{3\left(1 + \left\| F_{[N -1]} ( \bm{\alpha} ) \right\| \right)}\right)^n \\
&\leq - \log \left( 1 - \sum_{n =1}^{N -1}\left(\frac{1}{3}\right)^n \right) 
\ll 1, 
\end{align*}
which deduces $\| \bm{\alpha} \| \ll_N \left(1 + \left\| F_{[N -1]} ( \bm{\alpha} ) \right\| \right)^{N -1}$.
Hence, we have the first assertion of this lemma.

To prove the second assertion of this lemma, 
we will show 
\begin{equation}\label{eqn:1072}
\sum_{n=1}^{N-1} \left| \partial^{\underline{i}} F_n( \bm{\alpha}_{[n]} ) \right| 
\ll_N \left( 1 + \| \underline{F}_{[N -1]} (\bm{\alpha}) \| \right)^{(N -1)^2}
\end{equation}
for $\underline{i}\in \mathbb{N}_0^n$ and $| \underline{i}| \geq 1$.
By (ii) of Lemma \ref{lem:BE}, 
we have
\begin{align}
&\sum_{n=1}^{N-1} \left| \partial^{\underline{i}} F_n( \bm{\alpha}_{[n]} ) \right| X^n 
\unlhd f^{\abs}(X;\partial^{\underline{i}}\underline{F}(\underline{\alpha^{(0)}})) \label{eqn:1073} \\
\unlhd& X^{S(\underline{i} )} \exp\left( f^{\abs} (X; \underline{\alpha^{(0)}}) \right)
= X^{S(\underline{i})} \exp \left( \sum_{n =1}^{N-1} |\alpha_n| X^n \right). \nonumber
\end{align}
By evaluating \eqref{eqn:1073} at $X = \left(3 \left(1 + \left\|   \bm{\alpha}  \right\| \right) \right)^{-1}$, 
we have
\begin{align*}
\left(\frac{1}{3\left(1 + \left\|   \bm{\alpha}  \right\| \right)}\right)^{N -1} \sum_{n=1}^{N-1} \left| \partial^{\underline{i}} F_n( \bm{\alpha}_{[n]} ) \right| 
\leq& \sum_{n=1}^{N-1} \left| \partial^{\underline{i}} F_n( \bm{\alpha}_{[n]} ) \right| \left(\frac{1}{3\left(1 + \left\|   \bm{\alpha}  \right\| \right)}\right)^n \\
 \leq& \left(\frac{1}{3}\right)^{S(\underline{i})} \exp \left(  \sum_{n =1}^{N-1} \left(\frac{1}{3}\right)^n \right)
\ll 1,
\end{align*}
which implies $\sum_{n=1}^{N-1} \left| \partial^{\underline{i}} F_n( \bm{\alpha}_{[n]} ) \right| \ll_N \left(1 + \left\|   \bm{\alpha}  \right\| \right)^{N -1}$.
By (i) of Lemma \ref{lem:BE1}, 
we obtain
\[
\sum_{n=1}^{N-1} \left| \partial^{\underline{i}} F_n( \bm{\alpha}_{[n]} ) \right|
\ll_N \left( 1 + \| \underline{F}_{[N -1]} (\bm{\alpha}) \| \right)^{(N -1)^2},
\]
which gives the estimate \eqref{eqn:1072}.

We will show the second assertion. 
Let $\| (z_0, \bm{z}) - (\alpha_0, \bm{\alpha}) \| < \delta < 1$.
By the Taylor series expansion, we have
\begin{equation}\label{eqn:1074}
\left| e^{z_0} - e^{\alpha_0} \right|
= \left| e^{\alpha_0} \right| \left| e^{z_0 - \alpha_0} -1 \right|
\leq \left| e^{\alpha_0} \right| \sum_{n =1}^\infty\frac{\delta^n}{n!}
\leq 2 \left| e^{\alpha_0} \right| \delta.
\end{equation}
By the Taylor series expansion, by the equation \eqref{eqn:1072} and by using $\deg F_n = n$ , we have
\begin{align}
&F_n(\bm{z}_{[n]}) - F_n(\bm{\alpha}_{[n]}) \label{eqn:1075} \\
=& \sum_{\substack{ \underline{i} = (i_1, \ldots, i_{n}) \in \mathbb{N}_0^n;\\ 1 \leq | \underline{i} | \leq n }} \frac{  \partial^{\underline{i}} F_n ( \bm{\alpha}_{[n]} ) }{ \underline{i} ! } \left(\bm{z}_{[n]} - \bm{\alpha}_{[n]} \right)^{\underline{i}}
\ll_N \left( 1 + \| \underline{F}_{[N -1]} (\bm{\alpha}) \| \right)^{(N -1)^2} \delta \nonumber
\end{align}
for $1 \leq n \leq N -1$.
Hence we deduce
\[
\left\|  \underline{F}_{[N-1]} ( \bm{z}) -  \underline{F}_{[N-1]} ( \bm{\alpha} ) \right\|
\ll_N \left( 1 + \| \underline{F}_{[N -1]} (\bm{\alpha}) \| \right)^{(N -1)^2} \delta .
\]
Therefore we obtain
\begin{align*}
&\left\| e^{z_0} \underline{F}_{[N-1]} ( \bm{z} ) - e^{\alpha_0} \underline{F}_{[N-1]} ( \bm{\alpha} ) \right\|\\
\leq& | e^{z_0}  | \left\|  \underline{F}_{[N-1]} ( \bm{z} ) -  \underline{F}_{[N-1]} ( \bm{\alpha} ) \right\|
+ \left\| \underline{F}_{[N-1]} ( \bm{\alpha} ) \right\| | e^{z_0} - e^{\alpha_0} | \\
\leq& | e^{z_0} - e^{\alpha_0}  | \left\|  \underline{F}_{[N-1]} ( \bm{z} ) -  \underline{F}_{[N-1]} ( \bm{\alpha}) \right\| \\
&+ | e^{\alpha_0} | \left\|  \underline{F}_{[N-1]} ( \bm{z} ) -  \underline{F}_{[N-1]} ( \bm{\alpha} ) \right\| 
+ \left\| \underline{F}_{N-1} ( \bm{\alpha} ) \right\| | e^{z_0} - e^{\alpha_0} | \\
\ll&_N | e^{\alpha_0} | \left( 1 + \left\| \underline{F}_{[N -1]} (\bm{\alpha}) \right\| \right)^{(N-1)^2} \delta
\end{align*}
by the estimates \eqref{eqn:1074} and \eqref{eqn:1075}. 
This completes the proof.
\end{proof}

\begin{proof}[Proof of Corollary \ref{cor:CV}]
Let $\epsilon \in (0, 1)$, $\underline{c} = (c_k)_{k =0}^{N -1} \in \mathbb{C}^N$ with $| c_0 | \neq 0$.
Put
\[
z_0(t) 
=\log \LF ( \sigma_0 + it ),
\]
\[
\bm{z}(t)
= \left( \frac{1}{1!} \frac{d}{d s} \log \LF (\sigma_0 + it), \ldots, \frac{1}{(N -1)!}\frac{d^{N-1}}{ds^{N-1}} \log \LF (\sigma_0 + it)  \right),
\]
\[
\beta_1 = \frac{c_1}{c_0 \cdot 1!},~
\beta_2 = \frac{c_2}{c_0 \cdot 2 !}
\ldots,~
\beta_{N -1} = \frac{c_{N -1}}{c_0 (N -1)!}, \bm{\beta}= (\beta_k)_{k =1}^{N-1}
\]
and 
\[
\alpha_0 = \log c_0,~
\alpha_1 = G_1(\bm{\beta}_{[1]}), \ldots, \alpha_n = G_{N -1}(\bm{\beta}_{[N -1]}), \bm{\alpha} = (\alpha_j)_{j =1}^{N-1} . 
\]
Note that $\bm{\beta} = \underline{F}_{[N -1]} \left( \bm{\alpha}\right ) $ holds.
Let
\[
\delta = \delta ( \epsilon, \underline{c}, \Delta_N )
= \frac{ \Delta_N \cdot \epsilon }{ \left( 1 +  \left |e^{\alpha_0} \right)\right|  \left( 1 + \left\| \underline{F}_{[N -1]}(\bm{\alpha}) \right\| \right)^{(N-1)^2} } 
\in (0, 1), 
\]
where $\Delta_N$ is sufficiently small depending on $N$.
Note that, by the relation \eqref{eqn:rel}, we have
\[
e^{z_0(t)} \underline{F}_{[N -1]} \left( \bm{z}(t) \right)
=\left( \frac{1}{1!} \frac{d}{ds} \LF( \sigma_0 + it ), \ldots, \frac{1}{(N -1)!} \frac{d^{N-1}}{ds^{N-1}}\LF( \sigma_0 + it ) \right)
\]
for $\LF( \sigma_0 + it ) \neq 0$. 
Hence if
\begin{equation}\label{eqn:LLF}
\left\| \left( z_0(t), \bm{z}(t) \right) - \left( \alpha_0, \bm{\alpha} \right) \right\| < \delta 
\quad \textrm{and} \quad
\LF( \sigma_0 + it ) \neq 0
\end{equation}
hold, then we have
\[
\left| \frac{d^k}{ds^k} \LF( \sigma_0 + it ) - c_k \right| < \epsilon
\quad \textrm{for $k = 0,1, \ldots, N-1$}
\]
by Lemma \ref{lem:BE1} {\rm (ii)}.
To get the inequality \eqref{eqn:LLF},
it is enough to notice
\[
\left| \frac{d^k}{ds^k} \log \LF ( \sigma_0 + it ) - k ! \alpha_k \right| < \frac{\delta}{N}
\quad \textrm{for $k = 0,1, \ldots, N-1$}.
\]
By Lemma \ref{lem:BE1} {\rm (i)}, we have
\[
\left\| \left(\alpha_0, 1! \alpha_1, \ldots, (N -1)! \alpha_{N -1} \right) \right\| + \frac{N}{\delta}
\ll_{N} | \log c_0 | 
+ \left( \frac{ \| \underline{c} \| }{ | c_0 | } \right)^{(N-1)^2} \frac{1 + | c_0 |}{ \epsilon}.
\]
Combining with Theorem \ref{thm:MTH}, we have the conclusion.
\end{proof}

\subsection{Proof of Corollary \ref{cor:GLMSS}}
Although the proof is almost the same as in \cite{GLMSS2010}, 
we shall give the full details.
\begin{proof}[Proof of Corollary \ref{cor:GLMSS}]
Let the setting be as in Corollary \ref{cor:GLMSS}.
We will use the Taylor expansion series to prove the corollary.
Recall the Cauchy integral formula
\[
g^{(k)}(s_0)
= \frac{k!}{2 \pi i} \int_{| z - s_0 | = r} \frac{g(z)}{(z - s_0)^{k +1}}dz.
\]
Hence we have
\begin{align*}
|g^{(k)}(s_0) (s - s_0)^k |
\leq k ! M (g) \delta_0^k
\end{align*}
for $| s - s_0 | \leq \delta_0 r$.
By using the Taylor expansion series, 
we have
\begin{equation*}
\Sigma_1
:= \left| g (s) - \sum_{0 \leq k <N} \frac{g^{(k)}(s_0)}{k!} ( s - s_0 )^k \right|
\leq M (g) \sum_{k = N}^\infty \delta_0^k
= M(g) \frac{\delta_0^N}{1 - \delta_0}
\end{equation*}
for $ | s - s_0 | \leq \delta_0 r $.
We chose $N = N ( \delta_0, \epsilon, M(g))$ such that
\[
M(g) \frac{\delta_0^N}{1 - \delta_0} 
< \frac{\epsilon}{3}.
\]
We apply Corollary \ref{cor:CV} with $c_k = g^{(k)}(s_0)$ and $(\epsilon /3) \exp( - \delta_0 r )$ in place of $\epsilon$.
We chose $T = T(\LF, \sigma_0, g, \epsilon, \delta_0, N)$ such that
\[
T 
\geq \max\left\{ \exp\exp \left( C_2(\LF, \sigma_0, N) B \left( N, \mathbf{g}, (\epsilon/3)\exp( - \delta_0 r ) \right)^{d(\sigma_0, E_{\LF})} \right) , r \right\}.
\]
Then there exists $t_1 \in [T, 2T]$ such that
\[
\left| \LF^{(k)} ( \sigma_0 + it_1 ) - g^{(k)}( \sigma_0 + it_0 ) \right| < \frac{\epsilon}{3} \exp( - \delta_0 r )
\]
for $0 \leq k <N$.
Put $\tau: = t_1 - t_0$ and note that $\sigma_0 + i t_1 = s_0 + i \tau$ holds.
Remark that our choice of $T \geq r$ make the disc $\{ s ~;~ |s - s_0 | \leq \delta_0 r \} + i \tau$ avoid from including the pole of $\LF (s )$.
Hence we have
\begin{align*}
\Sigma_2
&:= \left| \sum_{0 \leq k < N} \frac{\LF^{(k)}(s_0 + i \tau)}{k!} ( s - s_0 )^k - \sum_{0 \leq k < N} \frac{ g^{(k)} ( s_0 ) }{ k !} (s - s_0)^k \right| 
\\
&< \frac{\epsilon}{3} \exp( - \delta_0 r ) \sum_{0 \leq k < N} \frac{( \delta_0 r )^k}{k!}
< \frac{\epsilon}{3} 
\end{align*}
for $| s - s_0 | \leq \delta_0 r$.
On the other hand, 
we have
\begin{equation*}
\Sigma_3
:= \left| \LF (s + i \tau) - \sum_{0 \leq k < N} \frac{\LF^{(k)} ( s_0 + i \tau ) }{k!} ( s - s_0 )^k \right| < M ( \tau ) \frac{\delta^N}{ 1 - \delta } \label{eqn:EA3}
\end{equation*}
for $ | s - s_0 | \leq \delta r $ and $0 < \delta \leq \delta_0$.
Therefore we have
\[
\left| \LF ( s + i \tau ) - g (s) \right|
\leq \Sigma_1 + \Sigma_2 + \Sigma_3
< \frac{2}{3} \epsilon + M ( \tau ) \frac{ \delta^N }{1 - \delta}
\]
for $ | s - s_0 | \leq \delta r $ and $0 < \delta \leq \delta_0$.
Choosing $\delta$ which satisfies
\[
M(\tau;\LF) \frac{\delta^N}{1 - \delta}
< \frac{\epsilon}{3}, 
\]
we have the conclusion.
\end{proof}

\begin{ackname}
The author is deeply grateful to Professor Kohji Matsumoto for many useful comments and suggestions.
The author would like to thank Prof.~ Tomohiro Ikkai, Dr.~ Yuta Suzuki, Dr.~ Wataru Takeda, Mr.~ Sh\={o}ta Inoue and Mr.~ Hirotaka Kobayashi for giving the author some good advise.
\end{ackname}

\begin{flushleft}
{\footnotesize
{\sc
Graduate School of Mathematics, Nagoya University,\\
Chikusa-ku, Nagoya 464-8602, Japan.
}\\
{\it E-mail address}, K. Endo: {\tt m16010c@math.nagoya-u.ac.jp}
}
\end{flushleft}
\end{document}